\documentclass[11pt]{elsarticle}

\usepackage{graphicx,epsfig,amssymb,amsbsy,amsmath,amsthm,hyperref,cleveref,setspace}
\usepackage{color}
\newcommand{\bb}[1]{\left({#1}\right)}					
\newcommand{\sq}[1]{\left[#1\right]}						
\newcommand{\cc}[1]{\left\{#1\right\}}					
\newcommand{\op}[1]{\mathcal{#1}}
\newcommand{\ord}[1]{{\sf O}\bb{#1}}					

\newcommand{\sfrac}[2]{\mbox{$\frac{#1}{#2}$}}	
\newcommand{\hf}{\mbox{$\frac12$}}

\newcommand{\surf}{{\cal D}}

\newcommand{\eps}{\varepsilon}

\newtheorem{theorem}{Theorem}
\newtheorem{lemma}[theorem]{Lemma}

\newtheorem{assumption}[]{Assumption}

\newcommand{\sign}{\operatorname{sign}}

\newcommand{\crefs}[2]{\cref{#1} and \cref{#2}}

\crefname{equation}{}{}
\Crefname{equation}{}{}

\def\eps{\varepsilon}

\journal{a journal}

\begin{document}

\begin{frontmatter}

\title{Novel slow-fast behaviour in an oscillator\\ driven by a frequency-switching force}

\author{Carles~Bonet\fnref{label1}}
\ead{carles.bonet@upc.edu}
\author{Mike R. Jeffrey\corref{cor1}\fnref{label4}}
\ead{mike.jeffrey@bristol.ac.uk}
\author{Pau~Mart\'in\fnref{label1,label2}}
\ead{p.martin@upc.edu}
\author{Josep~M.~Olm\fnref{label1,label3}}
\ead{josep.olm@upc.edu}

\cortext[cor1]{Corresponding author: Mike R. Jeffrey}
\address[label1]{Departament de Matem\`atiques, Universitat Polit\`{e}cnica de Catalunya, Spain}
\address[label2]{Centre de Recerca Matem\`atica, Barcelona, Spain}
\address[label3]{Institute of Industrial and Control Engineering, Universitat Polit\`{e}cnica de Catalunya, Spain}
\address[label4]{Department of Engineering Mathematics, University of Bristol, UK}
%
%

\date{\today}


\begin{abstract}
When an oscillator switches abruptly between different frequencies, there is some ambiguity in deciding how the system should be modelled at the switch. Here we describe two seemingly natural models of a switch in a simple periodically-forced harmonic oscillator, which disagree starkly in their predictions of its long time behaviour. Attempting to resolve the disagreement by `regularizing' the switch not only preserves the disagreement, but shows it increases with time. One of the models corresponds to a conventional `Filippov' description of a nonsmooth system, while the second exhibits a structure that irreversibly ages, developing a number of novel multi-scale behaviours that we believe have not been reported before. These include slow-fast staircases, novel mixed-mode oscillations, and a synchronized canard explosion. These features are proven to exist using asymptotic analysis, but as they involve a slow-fast time-scale separation that increases with time, they lie beyond the reach of numerical methods.
\end{abstract}

\begin{keyword}
nonsmooth, Filippov, slow-fast, timescale, ageing, switching, canard
\end{keyword}

\end{frontmatter}


\section{Introduction}

Take a damped oscillator, say $\ddot y=-a\dot y-y$, and apply a sinusoidal forcing that switches between two frequencies,
\begin{align}\label{sys0}
\ddot y=-a\dot y-y-\sin(\pi\omega_\pm t)\;,
\end{align}
such that the frequency $\omega_+$ applies during forward motion, $\dot y>0$, and $\omega_-$ applies during backward motion, $\dot y<0$. How should the switching of the system be modelled around the threshold $\dot y=0$?

This was the problem posed in \cite{j15hidden}, and its study uncovers deep unresolved issues for the general study of nonsmooth systems, as well as revealing new dynamical phenomena. A simplified first order analogue of \cref{sys0} was studied in \cite{j20malt}, corresponding to an electrical circuit switching between two different current sources. Here we extend this to study the much richer second order system \cref{sys0}, revealing new phenomena of dynamic ageing, fast-slow staircases, and synchronized canard explosions. With some preparatory work having been carried out in \cite{j20malt} to prove the more non-standard technical details, here we can explore these phenomena more qualitatively, avoiding lengthy technical formalities where possible.

The crux of the problem lies first in how we define \cref{sys0} at the switch between the two frequencies $\omega_\pm$, and then how we attempt to solve it. For simplicity let us fix the two frequencies as
\begin{align}\label{omega}
\omega_+=3/2\qquad{\rm and}\qquad \omega_-=1/2
\end{align}
(for the purposes of our main results any two values can be chosen, these were merely adopted in \cite{j20malt,j15hidden} as they provided clear illustrations).
Let us then write the sinusoidal forcing as a function $f(t,\lambda)$ in terms of some quantity $\lambda$ that represents a control mechanism, such that $f(t,\pm1)=\sin\bb{(1\pm\hf)\pi t}$, and for simplicity we take
\begin{align}\label{sign}
\lambda\in\begin{cases}
+1&{\rm if}\;\;z>+\eps\;,\\z/\eps&{\rm if}\;\;|z|\le\eps\;,\\-1&{\rm if}\;\;z<-\eps\;,
\end{cases}
\end{align}
for small $\eps\ge0$. A perfectly fast switch is obtained in the limit $\eps\rightarrow0$. If we replace this ramp function in \cref{sign} with any smooth strictly monotonic sigmoid, then our results are not changed substantially. If less idealised switching processes are taken into account, perhaps involving delays, hysteresis, or stochasticity, some indications of what may happen can be found in \cite{j18model}, but are beyond our scope here.

Written as a first order system, \cref{sys0} then becomes
\begin{align}\label{sys}
\dot x&=1\;,\nonumber\\
\dot y&=z\;,\\
\dot z&=-az-y-f(x,\lambda)\;.\nonumber
\end{align}
This still leaves us with the problem of finding a suitable expression for $f$ in terms of $\lambda$, i.e. the control switch. There is no unique solution to this, and our purpose here is to show how different expressions result in entirely incommensurable, and sometimes novel, dynamics.


We can use an elementary physical device consistent with \cref{sys0} to motivate the modeling of $f$. Take a current source in series with a resistor, inductor, and capacitor, as shown in the lefthand pictures in \cref{fig:rod}. Let the source $f$ be an alternating voltage, which switches between frequencies $\omega_\pm$ according to the direction of the current. Upon non-dimensionalization we directly obtain the model \cref{sys}, with charge $y$ and current $z=\dot y$. (This is also analogous to a mechanical scenario of a mass on a spring experiencing a sinusoidal driving force, as found in devices like solenoid valves, for example).
\begin{figure}[h!]\centering
\includegraphics[width=0.95\textwidth]{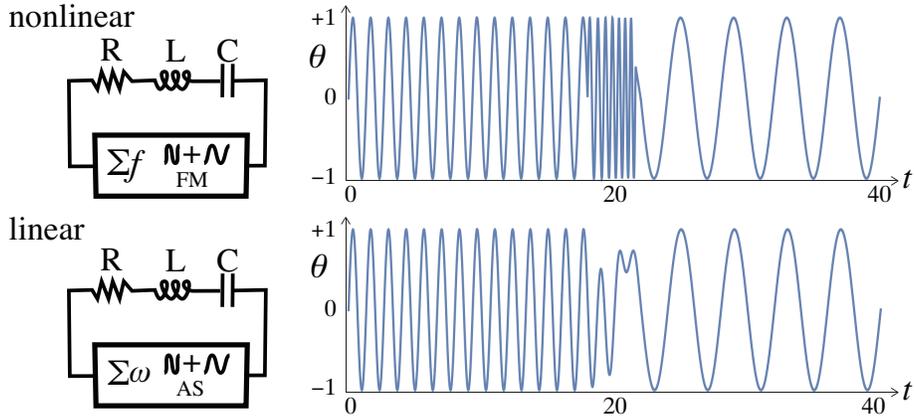}
\vspace{-0.3cm}\caption{\small\sf Left: sketch of the R-L-C circuit described in the text, with a nonlinear or linear modulation of a source $f=\sin(\pi\omega t)$. Top: frequency modulation between $\omega=1\pm\hf$. Bottom: additive synthesis of the waveforms $f=\sin\bb{\pi t(1\pm\hf)}$. Right: illustration of how these result in different signal profiles $f$ as the frequency switches according to \cref{non} (top) or \cref{lin} (bottom), using \cref{sign} for a toy solution $z(t)=t-20$ with $\eps=2$; note the contrast in frequency and amplitude modulation around $t=20\pm2$.}\label{fig:rod}\end{figure}
%

The form of $f$ comes from how we carry out the switch between voltage sources.
Given that $f(x,\pm1)=\sin\bb{(1\pm\hf)\pi x}$, perhaps the most obvious expression for $f(x,\lambda)$ is to write
\begin{subequations}\label{nonlin}
\begin{align}\label{non}
f(x,\lambda)=\sin\bb{(1+\hf\lambda)\pi x}\;,
\end{align}
which would be consistent with a {\it frequency modulator} that ramps abruptly between values $\omega=1+\hf\lambda$, as illustrated in the upper part of \cref{fig:rod}.

An alternative is to assume that the voltage source interpolates between the waveforms $f=\sin(\sfrac32\pi x)$ and $f=\sin(\hf\pi x)$ themselves, that is,
\begin{align}\label{lin}
f(x,\lambda)=\hf(1+\lambda)\sin(\sfrac32\pi x)+\hf(1-\lambda)\sin(\hf\pi x)\;,
\end{align}
\end{subequations}
more akin to {\it additive synthesis} of wave sources, illustrated in the lower part of \cref{fig:rod}.

Both of these are common methods for combining or switching between different waveforms in signal creation or musical applications.
On the right of \cref{fig:rod} we give an example of the different force profiles that the system may experience, as it transitions between the two frequencies when governed by these two different expressions. These graphs give a hint of why the expressions \cref{non} and \cref{lin} might influence the system differently, but not of how sever that difference will be. 

Our aim here will be to show how starkly the behaviours of these two models differ, despite them being indistinguishable in the limit $\eps\rightarrow0$, and both being consistent (among other possible models for $f$) with \cref{sys0}. The behaviour that results from \cref{lin} is consistent with Filippov's theory of nonsmooth systems \cite{f88}, so it has particular significance as a widely used convention, but being standard, we will summarize it only briefly. Instead we largely focus on the behaviours that result from \cref{non}, which seem to exhibit phenomena that have possibly not been seen before in the dynamical systems theory of either nonsmooth or slow-fast systems.

Both models \cref{nonlin} have two key features in the limit $\eps\rightarrow0$. Firstly, as we show in \cref{sec:periodics}, they share similar periodic orbits that pass through the regions $z\gtrless0$ and {\it cross} through $z=0$ transversally. Secondly, both models involve a region on $\cc{z=0,\;|y|\le1}$ where {\it sliding} occurs, that is, motion that slides along the switching threshold $z=0$. It is in the sliding motion that the two different models \cref{non} and \cref{lin} differ markedly. 

For $\eps$ small, both models \cref{nonlin} lead to an invariant manifold inside the switching layer $|z|<\eps$ that is responsible for the dynamics becoming constrained to {sliding} along $z=0$ in the limit $\eps\rightarrow0$. However, the form of those manifolds, and dynamics associated with them, differ entirely.

With the representation \cref{lin}, the hyperbolicity of the invariant manifolds scales with $1/\eps$, while with \cref{non} the hyperbolicity scales with $x/\eps$, a crucial difference. With \cref{lin}, the invariant manifold simply oscillates in time $x$, inducing simple sliding behaviour and simple local and periodic attractors consistent with the most commonly adopted convention of the {\it Filippov system} for \cref{sys0}; we introduce these briefly in \cref{sec:linear}. With the representation \cref{non} the system ages, i.e. changes irreversibly as $x$ increases, 
because the invariant manifold becomes increasingly packed together with a density $1/x$, growing new branches, leading to intricate (and to our knowledge new) multiple-scale phenomena involving mixed-mode oscillations, synchronized canard explosions, and a mechanism of fast-slow stepping across the switch. These various phenomena are introduced in the following sections.

We first analyse the linear system given by \cref{sys} with \cref{lin} in the switching layer $|z|\le\eps$ in \cref{sec:linear}, but only briefly, as this can be studied using standard concepts. From then on we focus on the nonlinear system in the switching layer $|z|\le\eps$, given by \cref{sys} with \cref{non}, which seems to exhibit behaviours unseen in either nonsmooth systems or in systems with slow-fast timescales. The elements of these behaviours are proven in lemmas in \cref{sec:slowfast}, and then, foregoing lengthy but trivial matching that we capture in an assumption at the start of \cref{sec:bigsmall}, we give our main theorems revealing novel long term behaviour. That behaviour is utterly incommensurable with the linear formulation of the system.

We make only some basic observations that apply to the behaviour outside the switching layer, in $|z|>\eps$, in the limit $\eps=0$ in \cref{sec:periodics}. We conclude with some final remarks in \cref{sec:conc}, including indications for the many directions of interest for future study.

Our objective is to highlight certain novel dynamical properties of \cref{non} that distinguish it from \cref{lin}, as an example of the potential differences of these and other possible formulations of a system that involves switching. To this end, we will not rigorously prove all results or seek to characterize every behaviour, only derive the distinguishing features and prove the less standard results. We will solely consider positive times $x>0$.


\section{The linear-switching system for $\eps\ge0$}\label{sec:linear}

To analyse the system \cref{sys} with the linear forcing \cref{lin} in the switching layer $|z|\le\eps$, let $z=\eps u$ and study the dynamics of $u$ on $[-1,+1]$. Some basic features, including two types of orbit and the limit $\eps\rightarrow0$, are sketched in \cref{fig:typesl}.
\begin{figure}[h!]\centering
\includegraphics[width=0.8\textwidth]{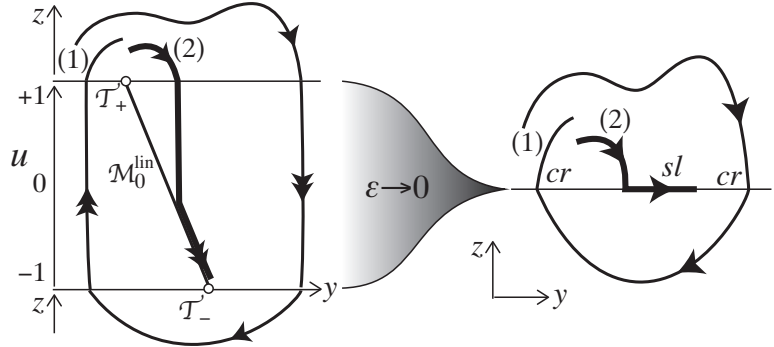}
\vspace{-0.3cm}\caption{\small\sf The switching layer $|u|\le1$ and its collapse to $z=0$ as $\eps\rightarrow0$. We see the critical slow manifold $\op M_0^{\rm lin}$, which touches the edge of the layer at $\op T_\pm$. Two example solutions are shown that (1) cross or (2) slide on $z=0$ for $\eps=0$. }\label{fig:typesl}\end{figure}

For $(x,y,u)$ on $u\in[-1,+1]$ this gives the equations
\begin{align}\label{linslow}
\dot x&=1\;,\nonumber\\
\dot y&=\eps u\;,\\
\eps\dot u&=-\eps au-y-\hf(1+u)\sin(\sfrac32\pi x)-\hf(1-u)\sin(\hf\pi x)\;.\nonumber
\end{align}
This is a slow-fast system that can be studied using geometric singular perturbation theory (e.g. \cite{f79,j95,st96}), and we shall only briefly outline the main features. First setting $\eps=0$ in \cref{linslow} we find there is a critical manifold $\op M_0^{\rm lin}$ given by
\begin{align}\label{linM0}
\op M_0^{\rm lin}&=\left\{\;(x,y,u)\in\mathbb R\times\mathbb R\times(-1,+1)\;:\;\right.\nonumber\\
			&\qquad\quad\left.y=-\hf(1+u)\sin(\sfrac32\pi x)-\hf(1-u)\sin(\hf\pi x)\;\right\}\;,
\end{align}
as illustrated in \cref{fig:Mlin}. The manifold $\op M_0$ loses normal hyperbolicity with respect to the $u$ dynamics where the normal to $\op M_0$ has no $u$-component, that is where $\partial\dot u/\partial u=0$, which occurs on sets
\begin{align}\label{linfold}
\op C_{mn}&=\cc{\;(x,y,u)\in\op M_0^{\rm lin}\;:\;x=2m-\hf n\;}\;,
\end{align}
for $m\in\mathbb Z$, with $n$ taking values $\cc{-,0,+}$. Between each adjacent branch of $\op C_{mn}$ with $x=0,\hf,\sfrac32,2,\sfrac52,...$, the manifold $\op M_0$ changes between being attracting or repelling with respect to the fast $u$ dynamics, on subsets of $\op M_0$ defined in fact by \cref{linstab}. 

\begin{figure}[h!]\centering
\includegraphics[width=0.65\textwidth]{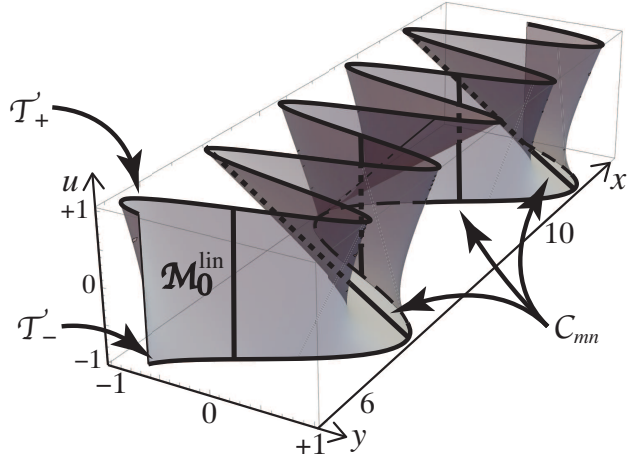}
\vspace{-0.3cm}\caption{\small\sf
The critical slow manifold $\op M_0$ in the switching layer for the linear system, which touches the edge of the layer at $\op T_\pm$, and has turning points at $\op C_{mn}$. }\label{fig:Mlin}\end{figure}

For $\eps>0$, in an $\eps$-neighbourhood of $\op M_0^{\rm lin}$ there exist invariant manifolds of slow dynamics, which we can express as
\begin{align}\label{Melin}
\op M_\eps^{\rm lin}=\cc{(x,y,u)\;:\;y=Y(x,u;\eps)}\;,
\end{align}
where
\begin{align}
y=Y(x,u;\eps):=Y_0(x,u)+\eps Y_1(x,u)+\eps^2 Y_2(x,u)+...
\end{align}
To find the functions $Y_r$ we solve
\begin{align}
0&=\sfrac{d\;}{dt}(Y-y)=(\dot x,\dot y,\dot u)\cdot\nabla (Y-y)\nonumber\\
&=	Y_{0,x}-Y_{0,u}( au+ Y_1)\nonumber\\&\quad
	+\eps\cc{ Y_{1,x}- u- Y_{0,u}Y_2- Y_{1,u}(au+ Y_1)}+\ord{\eps^2}\nonumber
\end{align}
which implies
\begin{align}
Y_0&=-\hf(1+u)\sin(\sfrac32\pi x)-\hf(1-u)\sin(\hf\pi x)\;,\nonumber\\
Y_1&=\sfrac{Y_{0,x}}{Y_{0,u}}-au=\sfrac\pi2\sfrac{3(1+u)\cos(\sfrac32\pi x)+(1-u)\cos(\hf\pi x)}{\sin(\sfrac32\pi x)-\sin(\hf\pi x)}-au\;,\\
Y_2&=-u/Y_{0,u}=2u/\cc{\sin(\sfrac32\pi x)-\sin(\hf\pi x)}\;.\nonumber
\end{align}
This gives the slow dynamics on $\op M_\eps^{\rm lin}$ as
\begin{align}\label{linslowe}
\dot x&=1\;,\nonumber\\
\dot y&=\eps u\;,\\
\dot u&=- \sfrac\pi2\sfrac{3(1+u)\cos(\sfrac32\pi x)+(1-u)\cos(\hf\pi x)}{\sin(\sfrac32\pi x)-\sin(\hf\pi x)}+\ord{\eps}\;,\nonumber
\end{align}
with solutions
\begin{align}\label{linsloworb}
x&=x_0+t\;,\nonumber\\
y&=y_0+\ord{\eps}\;,\\
u&=-\sfrac{2y_0+\sin(\sfrac32\pi x)+\sin(\hf\pi x)}{\sin(\sfrac32\pi x)-\sin(\hf\pi x)}+\ord{\eps}\;,\nonumber
\end{align}
where we must have $y_0=-\hf(1+u_0)\sin(\sfrac32\pi x_0)-\hf(1-u_0)\sin(\hf\pi x_0)$ by \cref{linM0}.

\begin{figure}[h!]\centering
\includegraphics[width=0.98\textwidth]{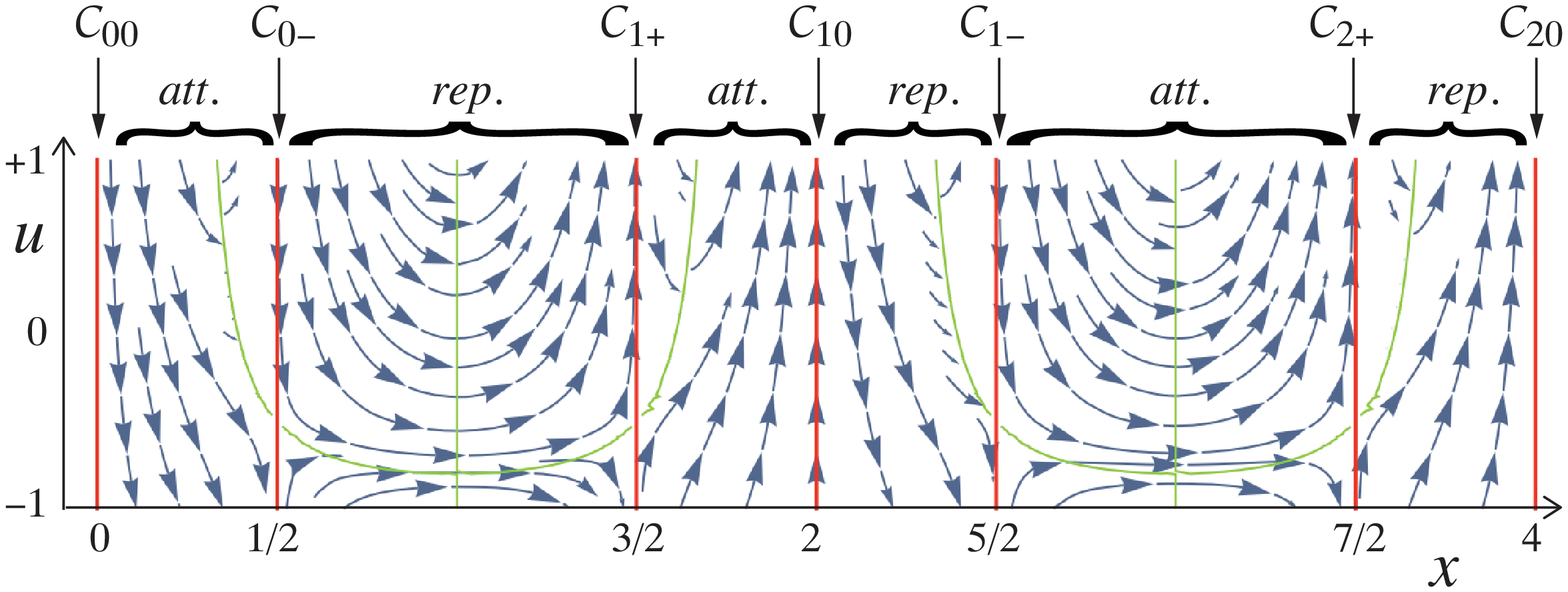}
\vspace{-0.3cm}\caption{\small\sf Dynamics on $\op M_\eps^{\rm lin}$, indicating its attracting ({\it att.}) and repelling ({\it rep.}) regions. }\label{fig:linslow}\end{figure}

Clearly these solutions are singular where the denominator vanishes, which corresponds to them reaching $\op C_{mn}$.  However, it is possible for solutions to pass through these sets between attracting and repelling branches of $\op M_\eps^{\rm lin}$. There is a distinguished non-singular trajectory that we can write to leading order as
\begin{align}\label{lincanard}
x&=2m+1+t\;,\nonumber\\
y&=-1/{\sqrt2}\;,\\
u&=\sfrac{ -\sqrt2-\cos(\sfrac32\pi t)+\cos(\hf\pi t)}{\cos(\sfrac32\pi t)+\cos(\hf\pi t)}\;,\nonumber
\end{align}
for $t\in(-\sfrac56,+\sfrac56)$, and for any $m\in\mathbb Z$. This is written such that it lies at $(x_0,y_0,u_0)=(2m+1,-\sfrac1{\sqrt2},-\sfrac1{\sqrt2})$ when $t=0$. It then passes between attracting and repelling branches of $\op M_\eps^{\rm lin}$ at $(x,y,u)=(2m+1\mp\hf,-\sfrac1{\sqrt2},-\hf)\in\op C_{m\pm}$. Everywhere else the sets $\op C_{mn}$ are impassable on $\op M_0^{\rm lin}$.



\medskip

We have only sketched some basics of the linear system, to perhaps be explored in future work, but sufficient to contrast to the nonlinear system in the remainder of the paper.

\bigskip
\section{The nonlinear switching system for $\eps\ge0$}\label{sec:slowfast}

Now take the system \cref{sys} with the nonlinear forcing \cref{non}, which we analyse in the switching layer $|z|\le\eps$ in much the same way as the linear system at the start of \cref{sec:linear}. Let $z=\eps u$ and study the dynamics of $u$ on $[-1,+1]$. Some basic features, including two types of orbit and the limit $\eps\rightarrow0$ are sketched in \cref{fig:types}.
\begin{figure}[h!]\centering
\includegraphics[width=0.8\textwidth]{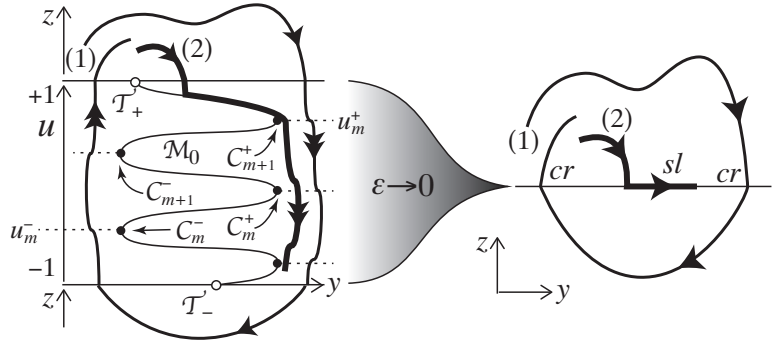}
\vspace{-0.3cm}\caption{\small\sf
The switching layer $|u|\le1$ and its collapse to $z=0$ as $\eps\rightarrow0$ (compare with \cref{fig:typesl} for the linear system). We see the critical slow manifold $\op M_0$, which touches the edge of the layer at $\op T_\pm$, and has turning points at $\op C_m^\pm$. Two example solutions are shown that (1) cross or (2) slide on $z=0$ for $\eps=0$. }\label{fig:types}\end{figure}

For $(x,y,u)$ on $u\in[-1,+1]$ this gives the equations
\begin{align}\label{slow}
\dot x&=1\;,\nonumber\\
\dot y&=\eps u\;,\\
\eps\dot u&=-\eps au-y-\sin\sq{\pi x(1+\hf u)}\;.\nonumber
\end{align}
Like \cref{linslow}, we obtain a slow-fast system on $u\in[-1,+1]$ with a critical manifold $\op M_0$, now given by
\begin{align}\label{M0}
\op M_0&=\cc{\;(x,y,u)\in\mathbb R^+\times\mathbb R\times(-1,+1)\;:\;y=-\sin\sq{\pi x(1+\hf u)}\;}\;,
\end{align}
but the righthand side of \cref{slow} now depends on $x$ in a manner that makes this unusual as a slow-fast system.

The manifold $\op M_0$ has curves of turning points (where the normal vector to $\op M_0$ has no $u$-component), on sets
\begin{subequations}\label{fold}
\begin{align}
\op C^+_m&=\cc{(x,y,u)\in\op M_0\;:\;y=+1,\;u=u_{m}^+:=\sfrac{4m-1}x-2\;}\;,\label{fold+}\\
\op C^-_m&=\cc{(x,y,u)\in\op M_0\;:\;y=-1,\;u=u_{m}^-:=\sfrac{4m+1}x-2\;}\;,\label{fold-}
\end{align}
\end{subequations}
for $m\in\mathbb Z$.
Notice that as $x$ increases the turning points of $\op M_0$ pack together with a density $2/x$, as given by \cref{fold} (with `density' meaning the distance between adjacent curves of turning points on either $y=+1$ or $y=-1$).

The critical manifold $\op M_0$ and turning points $\op C_m^\pm$ are illustrated in \cref{fig:Mnon}. Each branch of turning points indicated by $\op C_m^\pm$ in \cref{fig:types}, \cref{fig:Mnon}, and later figures, is for a different value of the index $m$.
\begin{figure}[h!]\centering
\includegraphics[width=0.75\textwidth]{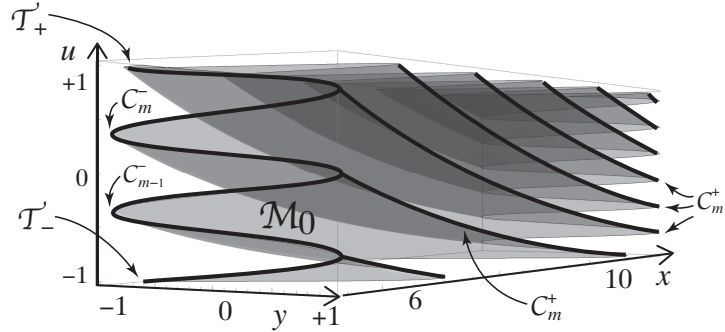}
\vspace{-0.3cm}\caption{\small\sf The critical slow manifold $\op M_0$ in the switching layer for the nonlinear system (compare with \cref{fig:Mlin} for the linear system), which touches the edge of the layer at $\op T_\pm$, and has turning points at $\op C_m^\pm$.
}\label{fig:Mnon}\end{figure}

For the linear system in \cref{sec:linear} we proceeded at this point by analysing its slow-fast dynamical separation on timescales $t$ and $t/\eps$. For the nonlinear system, by contrast, the derivative $\partial\dot u/\partial u$ from \cref{slow} grows unboundedly with $x$, and as a consequence we shall see below that the fast dynamics actually unfolds on a timescale $tx/\eps$.

For this reason we will primarily be concerned with what happens to the dynamics in the layer as $x/\eps\rightarrow\infty$. In the remainder of this section we calculate approximations for the main features of the slow and fast dynamics, which we will bring together in \cref{sec:bigsmall} to describe the different large and small forms of relaxation oscillations in the layer. To avoid the necessary calculations obscuring the simplicity of the results, we first summarize the main features informally in \cref{sec:brief}.
The dynamics in $u\in[-1,+1]$ will consist of slow components inside an $\eps$-neighbourhood of $\op M_0$, and fast components outside it, which we derive in \cref{sec:slowapprox} and \cref{sec:stairs}. 

\bigskip
\subsection{Brief outline of the system's dynamics}\label{sec:brief}

Let us briefly summarize the dynamics that will be more fully described in the following sections.

Inside the layer, the slow-fast system \cref{slow} exhibits complex multi-scale dynamics, composed of transitions between the following:
\begin{enumerate}
\item large cycles to be described in \cref{sec:bigcycles}, which could be considered a form of relaxation oscillations, comprised of large arcs in the slow dynamics (\cref{thm:bigarc}), and fast transitions that take the form of `staircases' (\cref{thm:stairs}), illustrated in \cref{fig:bigcyc}.
\item small cycles to be described in \cref{sec:smallcycles}, which consists of small arcs in the slow dynamics (to be derived in \cref{thm:smallarc}) and fast `staircase' transitions (\cref{thm:stairs}), illustrated in \cref{fig:smallcyc}.
\end{enumerate}

\begin{figure}[h!]\centering
\includegraphics[width=0.98\textwidth]{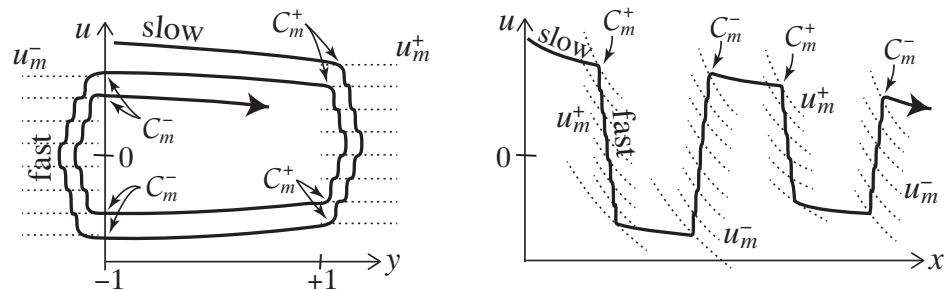}
\vspace{-0.3cm}\caption{\small\sf Large cycles inside the region $|u|<1$, shown in the $(u,y)$ and $(u,x)$ projections, consisting of slow arcs and fast staircases. As in \cref{fig:smallcyc} the dashed curves $u_m^\pm$ are the loci of $u$-coordinates of the turning points $\op C_m^\pm$. Each branch of turning points indicated by $\op C_m^\pm$, and corresponding values of $u$ indicated by $u_m^\pm$, is for a different value of $m$.}\label{fig:bigcyc}\end{figure}

\begin{figure}[h!]\centering
\includegraphics[width=0.9\textwidth]{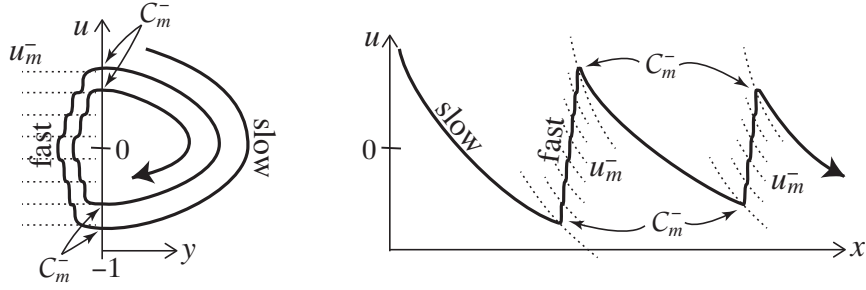}
\vspace{-0.3cm}\caption{\small\sf Small cycles inside the region $|u|<1$, again shown in the $(u,y)$ and $(u,x)$ projections, consisting of slow small arcs and fast staircases. The dashed curves $u_m^-$ are the loci of $u$-coordinates of the turning points $\op C_m^-$. }\label{fig:smallcyc}\end{figure}

Essentially our main results, which follow, involve deriving the elements that make up these large and small cycles.

\newpage
\subsection{Slow dynamics for large $x$: large and small arcs}\label{sec:slowapprox}

For large $x$ in \cref{slow}, it becomes impossible to separate out the fast oscillation of the $\sin(\pi x(1+\hf u))$ term from the slow-fast timescale separation created by small $\eps$. Hence the slow dynamics depends not just on the small parameter $\eps$, but also crucially on the largeness of the time $x$. Introducing a new variable
\begin{align}
v=xu\;,
\end{align}
the system \cref{slow} becomes
\begin{align}\label{slowv}
\dot x&=1\;,\nonumber\\
\dot y&=\sfrac{\eps}{x}v\;,\\
\sfrac{\eps}{x}\dot v&=\sfrac{\eps}{x}(\sfrac1x-a)v-y-\sin\sq{\pi (x+\hf v)}\;,\nonumber
\end{align}
making it explicit that the small quantity separating fast and slow dynamics in the system is $\eps/x$, rather than $\eps$ alone. The fast component of the sinusoid is now separated out from the variable $v$ in the argument $\pi (x+\hf v)$. 

For this system we will show the following.

\begin{lemma}[Slow manifolds]\label{thm:Me}
The slow dynamics of \cref{slow} lies on invariant manifolds in an $\sfrac\eps x$-neighbourhood of $\op M_0$ given by
\begin{align}\label{Me}
\op M_\eps=\cc{(x,y,v)\;:\;y=Y(x,v;\eps)}\;,
\end{align}
where
\begin{align}\label{Ye}
Y(x,v;\eps)&:=-\sin\sq{\pi (x+\hf v)}+\sfrac\eps{x}\bb{2+(\sfrac1x-a)v}+\ord{(\sfrac{\eps}{x})^2}\;.
\end{align}
\end{lemma}

An alternative to \cref{Me} to represent the slow manifolds $\op M_\eps$ is to express them as
\begin{align}
\op M_\eps=\cc{(x,y,v)\;:\;v=V_m(x,y;\eps)}\;,
\end{align}
where
\begin{align}
V_m(x,y;\eps):&=\theta_m(x,y)+\sfrac{2\eps}{\pi x}\sfrac{2-a\theta_m(x,y)}{\sqrt{1-y^2}}+\ord{\sfrac{\eps^2}{x^2}}\quad{\rm and}\nonumber\\
\theta_m(x,y)&:=-2\cc{x+(-1)^m\bb{m+\sfrac{1}{\pi}\arcsin y}}\;,
\end{align}
which makes it explicit that the slow manifolds have different branches, indexed by $r\in\mathbb Z$, with odd[even] $r$ giving [un]stable manifolds, separated by the set $\op C_m^\pm$ as seen in \cref{fig:Mnon}. Either form lead to the following.

\begin{lemma}[Slow dynamics]\label{thm:Me_arcs}
The dynamics on $\op M_\eps$ is given by
\begin{align}\label{M0proj}
\dot x&=1\;,\nonumber\\
\dot y&=\sfrac\eps{x}v\;,\\
\dot v&=-2+\ord{\sfrac\eps{x}}\;,\nonumber
\end{align}
whose solutions are arcs given by
\begin{align}
 y(x)&=y_0+2\eps(x_0-x)+\eps(v_0+2x_0)\log\frac{x}{x_0}+\ord{\eps/x_0}\;,\nonumber\\
 v(x)&=v_0+2(x_0-x)+\ord{\eps/x_0}\;,\label{arcapprox}
\end{align}
through any initial point $(x_0,y_0,v_0)\in\op M_\eps$.
\end{lemma}

The proofs of \cref{thm:Me} and \cref{thm:Me_arcs} follow by standard perturbation analysis and are given in Appendix~A. 
In fact one can show that the slow manifolds lie $\eps^2/x^2$ close to $\op M_0$ (rather than just $\eps/x$), outside the $\eps/x$-neighbourhood of the turning points; the result is easy to prove, as done in \cite{j20malt} for the simplified form of this system, but we will not reproduce it here.

At any instant $x$, the slow flow on $\op M_\eps$ given by \cref{thm:Me_arcs} forms a family of parabolic arcs in the $(y,u)$ plane.
Eliminating $x$ in \cref{arcapprox}, these arcs can be written as
\begin{align}\label{arcapproxyu}
 y&=y_0+\eps(v-v_0)+\eps(v_0+2x_0)\log\bb{1+\sfrac{v_0-v}{2x_0}}+\ord{\eps/x_0}\;.
\end{align}
The approximation \cref{arcapprox} (and hence \cref{arcapproxyu}) captures the slow arcs of the exact flow \cref{slow} (or \cref{slowv}) increasingly well as $x$ increases, as shown in \cref{fig:arcs}.

\begin{figure}[h!]\centering
\includegraphics[width=0.98\textwidth]{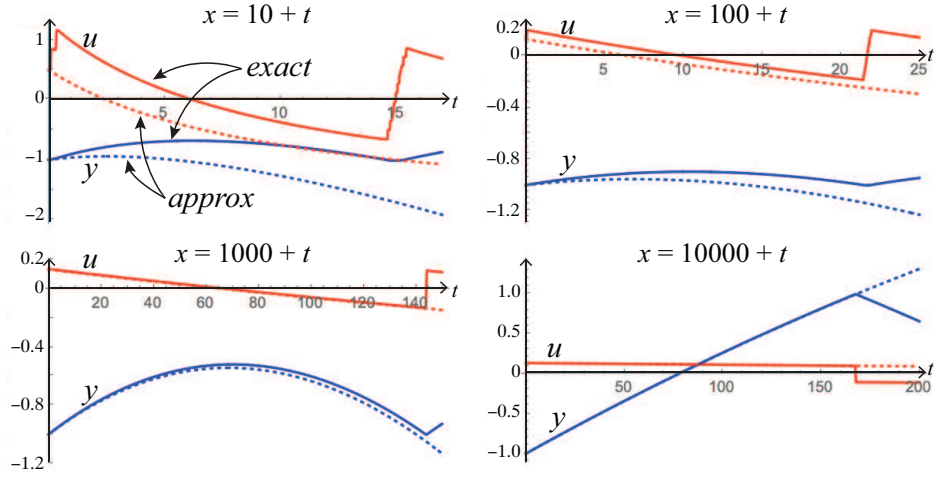}
\vspace{-0.3cm}\caption{\small\sf Plots of $y(t)$ and $u(t)$, full curves show the exact solution, and dotted curves show the slow approximation \cref{arcapprox}, which improves as $x$ increases. In each case we simulate just beyond the time when the trajectory passing a turning point and so enters the fast dynamics (at which the approximation fails). The first 3 plots show what we introduce in \cref{thm:smallarc} as {\it small arcs}, the last shows what we introduce in \cref{thm:bigarc} as a {\it large arc}. }\label{fig:arcs}\end{figure}

The arcs \cref{arcapprox} are confined to $\op M_\eps$ in the regions where the manifold is hyperbolically stable or unstable, but they can enter or leave $\op M_\eps$ at its turning points $\op C_m^\pm$. From the phase portrait in \cref{fig:smallcycleflow} it is obvious that these arcs take two forms: either they begin and end in a neighbourhood of $y=-1$ to form what we call {\it small arcs}, or they begin and end in neighbourhoods of $y=-1$ and $y=+1$, respectively, to form what we call {\it large arcs}.
\begin{figure}[h!]\centering
\includegraphics[width=0.4\textwidth]{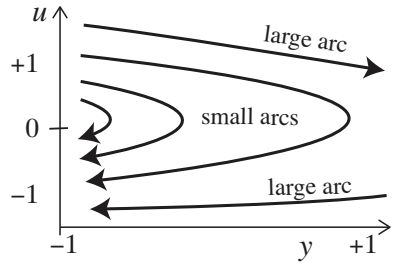}
\vspace{-0.3cm}\caption{\small\sf The projected slow flow consisting of arcs \cref{arcapprox}.
}\label{fig:smallcycleflow}\end{figure}

To describe these it is better to return to using the coordinate $u=v/x$.
To find the change in $x$ along these arcs from \cref{arcapprox} it is necessary to expand the logarithm, and while the most obvious route to do this is to approximate for large $x_0$, this must be treated with some care. As the following result shows, the time taken to traverse a small arc is $\delta x=\ord{x_0}$, which prevents a large $x_0$ approximation, and instead it is necessary to approximate for small $u_0$.

\begin{lemma}[Small arcs]\label{thm:smallarc}
The orbit through an initial point $(x_0,y_0,u_0)\in\op M_\eps$, with $u_0>0$ and
\begin{align}\label{arclim}
u_0^2< u_{0,c}^2&:=\frac{4(1-y_0)}{\eps x_0}+\ord{\sfrac{\eps}{x_0},\eps x_0u_0^3}\;,
\end{align}
passes through $u=0$ at a coordinate $(x_1,y_1,0)$ satisfying
\begin{align}\label{harcu}
 x_1&=x_0(1+\hf u_0)+\ord{\eps/x_0}\;,\nonumber\\
 y_1&=y_0+\sfrac14\eps x_0u_0^2+\ord{\eps/x_0,\eps x_0u_0^3}\;,
\end{align}
before returning to $y=y_0$ at a coordinate $(x_2,y_0,u_2)$ satisfying
\begin{align}\label{smallarcs}
%
x_2&=x_0(1+u_0+\sfrac16u_0^2+\ord{u_0^3})\cc{1+\ord{\eps/x_0}}\;,\nonumber\\
 u_2
 	&=(-u_0+\sfrac23u_0^2+\ord{u_0^3})\cc{1+\ord{\eps/x_0}}\;.
\end{align}
\end{lemma}

\begin{proof}
This follows from \cref{arcapprox} by integrating from an initial point 0 on $\op M_\eps$ with $u_0>0$, to a point 1 on $u_1=0$, and back to a point 2 with $u_2<0$, then expanding perturbatively for small $u_0$ in the form $x_i-x_0=c_{i,0}+c_{i,1}u_0+c_{i,2}u_0^2+...$ for $i=1,2$. The condition \cref{arclim} is necessary for these to exist, the bounding case $u_0=u_{0,c}$ being the solution of \cref{harcu} with $y_1=+1$, beyond which integrating forward from $(x_0,y_0,u_0)$ results in the orbit leaving $\op M_\eps$ before reaching $v=0$ (by reaching $\op C_m^+$ for some $m$).
\end{proof}

The initial condition $y_0$ satisfies $|y_0|\le1$ subject to \cref{harcu}, with the largest such arc being described by taking $y_0=-1$, for which $u_{0,c}^2=\frac{8}{\eps x_0}+\ord{\sfrac{\eps}{x_0},\eps x_0u_0^3}$.

This shows that the arc from $(x_0,y_0,u_0)$ to $(x_2,y_0,u_2)$ is almost symmetric about $u=0$, but the arcs are weakly contracting with order $u_0^2$.

If \cref{arclim} does not hold then an orbit will reach $\op C_m^+$ (for some $m\in\mathbb Z$) before it can turn around to complete a small arc, and instead create a large arc. To find these we approximate \cref{arcapprox} for large $x_0$.

\begin{lemma}[Large arcs]\label{thm:bigarc}
There exist orbits that, in $v>0$, connect pairs of points $(x_0,y_0,v_0)\in\op C_m^-$ to points $(x_1,y_1,v_1)\in\op C_m^+$, that is with $y_0=-y_1=-1$,
and respectively there exist orbits that, in $v<0$, connect pairs of points $(x_0,y_0,v_0)\in\op C_m^+$ to points $(x_1,y_1,v_1)\in\op C_m^-$, that is with $y_0=-y_1=+1$, in both cases satisfying
\begin{align}\label{bigarcs}
x_1&= x_0+\sfrac{2}{|u_0|\eps} +\ord{x_0^{-2}}\;,\nonumber\\
u_1&=u_0-\sfrac{2(2+u_0)}{|u_0|\eps x_0}+\ord{x_0^{-2}} \;,
\end{align}
provided $(x_0,u_0)$ does not satisfy \cref{arclim}.
\end{lemma}

\begin{proof}
Again this follows from \cref{arcapprox} by direct calculation, integrating for $v>0$ from an initial point 0 on $\op C_m^-$ with $y=-1$ to a point 1 on $\op C_m^+$ with $y=+1$, or for $v<0$ from point 0 on $\op C_m^+$ with $y=+1$ to point 1 on $\op C_m^-$ with $y=-1$, and approximating for large $x_0$. The condition \cref{arclim} must be violated for these to exist, otherwise integrating forward from $(x_0,y_0,u_0)\in\op C_m^-$ in the first case, or backward from $(x_1,y_1,u_1)\in\op C_m^-$ in the second case, results in the orbit reaching $v=0$ and forming a closed arc such that it never reaches $\op C_m^+$.
\end{proof}

We return to these two lemmas in \cref{sec:bigsmall} to describe how large and small arcs form cycles, connected by the fast dynamics outside $\op M_\eps$.

\bigskip
\subsection{Fast dynamics for large $x$: staircases}\label{sec:stairs}

To obtain the fast dynamics outside the $\eps$-neighbourhood of $\op M_0$, we again introduce the coordinate $v=ux$ to obtain the system \cref{slowv}. In typical slow-fast analysis following e.g. \cite{f79,j95}, one then changes to a fast timescale $\tau=t\eps/x$, but clearly the $x$-dependence of this timescale brings complications. In fact we can proceed more directly.

The fast dynamics that occurs in the region $|y|\le1$ between the different branches of $\op M_0$ is indeed described well by changing timescale. For this one may fix some $x_0$ and rescale time as $\tau=t\eps/x_0$, and analyse the fast dynamics for small times $t=\ord{\eps/x_0}$ around this, as $x=x_0+t$. Away from the turning points $\op C_m^\pm$, these trajectories will simply connect the attracting and repelling branches of $\op M_\eps$ to order $\ord{\eps/x_0}$.

More interesting is what happens in $|y|>1$, in particular in the region $|y|=1+\ord{\eps/x_0}$. Preliminary simulations reveal that the fast orbits do not pass through the layer $u\in[-1,+1]$ as simple straight lines as $\eps\rightarrow0$, but rather forms steps wherever they encounter a minimum of $\dot v$ (or $\dot u$) as represented by the double-arrow trajectories in \cref{fig:types}.

To see why these steps happen,
observe from \cref{slowv} that typically for $|y|>1$ we have $\sfrac{\eps}{x_0}\dot v=\ord{1}$, hence $\dot v=\ord{x_0/\eps}$ is large, so the dynamics of $v$ (and similarly $u$) is fast compared to the slower $\dot x=\ord1$ and very slow $\dot y=\ord{\eps/x_0}$. However, if we consider a point near $\op C_m^\pm$ such that $y=\pm1+\ord{\eps/{x_0}}$ and $v=x_0u_m^\pm+\ord{\sqrt{\eps/{x_0}}}={4m\mp1}-2x_0+\ord{\sqrt{\eps/{x_0}}}$, then
$$\sfrac{\eps}{x}\dot v=\sfrac{\eps}{x}(\sfrac1x-a)v+\ord{\sfrac\eps{x_0}}	\quad\Rightarrow\quad	
	\dot v=(\sfrac1x-a)v+\ord{1}$$
is no longer large with respect to $\eps/x_0$, and the dynamics of $v$ is on a comparable scale to $x$.

This leads to stepping as described by the following.
\begin{lemma}[Fast stepping]\label{thm:stairs}
For $|y|=1+\ord{\eps/x}$, solutions of \cref{sloww} map between adjacent planes $u=u_m^\pm$ for $m\in\mathbb Z$, according to
\begin{align}\label{stairmap}
x_k&=x_{k-1}+T_{k-1}\;,\nonumber\\
y_k&=y_{k-1}+\sfrac{\eps}{x_0} v_{k-1}T_{k-1}+\ord{\sfrac{\eps^3}{x_0^3}}\\
v_k&=v_{k-1}-4Y_\pm\;,\nonumber
\end{align}
where
\begin{align}\label{Tmap}
T_{k-1}=\frac{4\eps}{x_0\sqrt{(y_{k-1}+\sfrac{\eps}{2x_0} av_0)^2-1}}
\qquad{\rm and}\qquad Y_\pm:=\pm1\;.
\end{align}
\end{lemma}

The proof of this is not standard, though quite brief, so we include it below. Before the proof, we illustrate the approximation in \cref{fig:stairs}, plotting an exact solution of \cref{slowv} and the points obtained by iterating the map \cref{stairmap}.

\begin{figure}[h!]\centering\includegraphics[width=\textwidth]{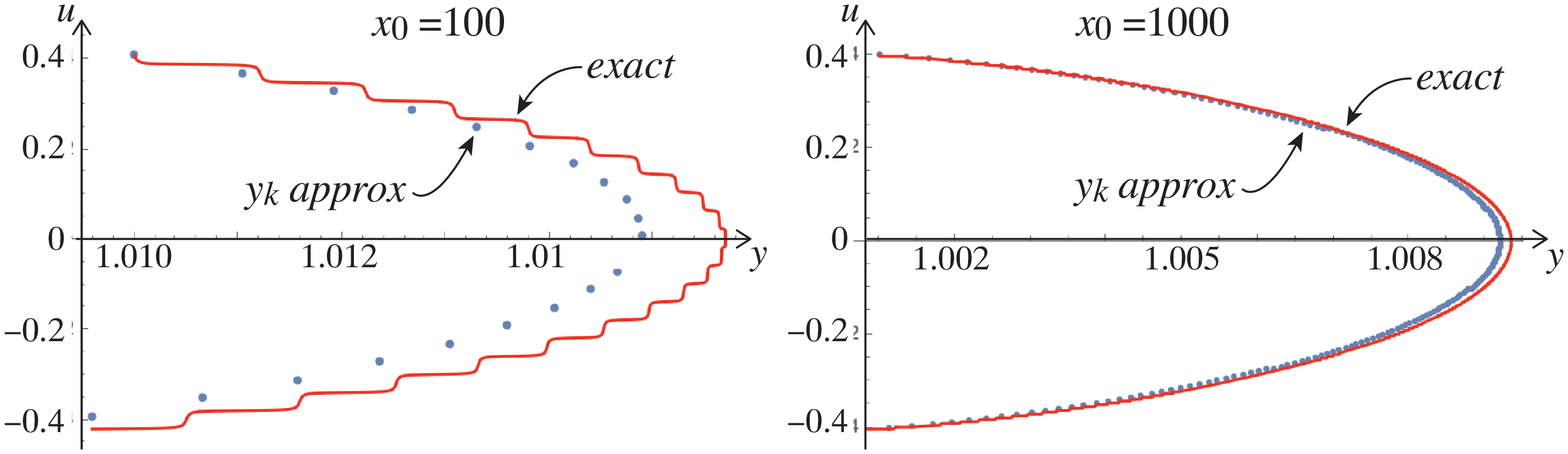}
\vspace{-0.3cm}\caption{\small\sf Staircases and their approximation for $\eps=0.1$. (The accuracy of the approximation improves with $x_0$ but is independent of $\eps$).
}\label{fig:stairs}\end{figure}

Notice that, if an orbit begins at some $(x_0,+1,v_0)$, the distance between turning points is of order $1/x_0$, so the number of turning points traversed in a staircase is of order $v_0=u_0x_0$, hence the ($\eta$ or) $y$-amplitude of a staircase is of size
\begin{align}
y_k-y_0\sim\sfrac{\eps^2v_0}{x_0^2}\times v_0=\sfrac{\eps^2}{x_0^2}v_0^2=\eps^2u_0^2\;.
\end{align}
Furthermore, over a staircase like those in \cref{fig:stairs}, taken from some $(x_0,y_0,u_0)$ and returning to some $(x_k,y_0,u_k)$, we have $x_k=x_0+\ord{\eps/x_0}$, so $x_0$ is approximately fixed over the staircase, leading to a symmetric number of steps outward to those inward, at least to order $\eps/x_0$. Hence over a full staircase from some $y_0$ in a neighbourhood of $y_0=\pm1$ and back to the same $y_k=y_0$ we have $v_k=-v_0+\ord{\eps/x_0}$ or equivalently
\begin{align}\label{wholestairs}
u_k=-u_0+\ord{\eps/x_0}\;.
\end{align}

\begin{proof}[Proof of \cref{thm:stairs}]

To analyse the dynamics in this region, define local coordinates $(\eta,\omega)$ near a turning point $(x_0,y_0,v_0)\in\op C_m^\pm$~, by letting
\begin{align}
y=(1+\eta)Y_\pm\;,\qquad v=v_0+\sfrac2{\pi}\omega\;,
\end{align}
where
\begin{align}
-\sin\sq{\pi (x_0+\hf v_0)}=Y_\pm:=\pm1\;.
\end{align}
The sinusoidal term in \cref{slowv} then becomes
\begin{align*}
\sin\left(\pi x_0\left(1+\sfrac{u}{2}\right)\right)&=\sin\sq{\sfrac{\pi}{2}(v-v_0)+\pi(x_0+\hf v_0)}\\
&=-Y_\pm\cos\sq{\sfrac{\pi}{2}(v-v_0)}\;,
\end{align*}
and so \cref{slowv} becomes
\begin{align}\label{sloww}
\dot x&=1\;,\nonumber\\
\dot \eta &=\sfrac{\eps}{x_0} (v_0+\sfrac{2}{\pi}\omega)\;,\\
\sfrac{2\eps}{\pi x_0} \dot \omega &=-\sfrac{\eps a}{x_0}(v_0+\sfrac{2}{\pi}\omega)-(1+\eta)Y_\pm+Y_\pm\cos\omega\;.\nonumber
\end{align}
%
%
In a time $\Delta t=T=\ord{\eps/x_0}$ we have $\eta=\eta_0+\sfrac{\eps}{x_0} v_0T\cc{1+\ord{\eps/x_0}}$, therefore $\eta-\eta_0=\ord{\eps^2/x_0^2}$, so we can approximate \cref{sloww} from an initial point $(x_0,Y_\pm\pm\eta_0,u_0)$, as
\begin{align}\label{sloww0}
\dot x&=1\;,\nonumber\\
\dot y &=\sfrac{\eps}{x_0} v_0+\ord{\eps^2/x_0^2}\;,\\\nonumber
\sfrac{2\eps}{\pi x_0} \dot \omega &=-\sfrac{\eps}{x_0} av_0-(1+\eta_0)Y_\pm+Y_\pm\cos\omega+\ord{\eps^2/x_0^2}\;.
\end{align}
We can use this to find the distance travelled in $y$ (equivalently $\eta$) from a point with coordinate $u=u_{m}^\pm$ to the subsequent $u_{m\mp1}^\pm$. We have just to integrate, 
\begin{align}
-\sfrac{\pi x_0}{2\eps}\int_0^T dt &=\int_0^{-2\pi Y_\pm} \frac{d\omega}{\sfrac{\eps}{x_0} av_0+(1+\eta_0-\cos\omega)Y_\pm}+\ord{\sfrac{\eps^2}{x_0^2}}\;,
\end{align}
which is solved with a simple use of Cauchy's residue theorem. First rescale $\omega=-Y^\pm\omega'$, then substitute $\psi=e^{i\omega}$, giving
\begin{align}
\sfrac{\pi x_0}{2\eps}T
&=-2\int_0^{2\pi} \frac{d\omega'}{-\sfrac{\eps}{x_0} av_0Y_\pm-2-2\eta_0+e^{i\omega'}+e^{-i\omega'}}+\ord{\sfrac{\eps^2}{x_0^2}}\nonumber\\
&=2i\int_0^{2\pi} \frac{d\psi}{\psi^2-2(1+\eta_0+\sfrac{\eps}{2x_0} av_0Y_\pm)\psi+1}+\ord{\sfrac{\eps^2}{x_0^2}}\;.
\end{align}
The integrand has poles at $\psi_\pm=1\pm\sfrac{\eta_0}{Y_\pm}\pm\sqrt{(1+\eta_0+\sfrac{\eps}{2x_0} av_0Y_\pm)^2-1}$, and only the $\psi_-$ pole is inside the unit circle, with residue $\frac{1}{\psi_--\psi_+}$, so we have
\begin{align}
\sfrac{\pi x_0}{2\eps}T
&=2i\times2\pi i\times\frac{1}{\psi_--\psi_+}+\ord{\sfrac{\eps^2}{x_0^2}}\nonumber\\
&=-4\pi \frac{-1}{2\sqrt{(1+\eta_0+\sfrac{\eps}{2x_0} av_0Y_\pm)^2-1}}+\ord{\sfrac{\eps^2}{x_0^2}}\nonumber\\
&=\frac{2\pi}{\sqrt{(y_0+\sfrac{\eps}{2x_0} av_0)^2-1}}+\ord{\sfrac{\eps^2}{x_0^2}}\;.
\end{align}
Hence the time taken to complete a step between two planes $u_{n,\pm}$ for adjacent $n$ and $n+1$ is
\begin{align}
T&=\frac{4\eps}{x_0\sqrt{(y_0+\sfrac{\eps}{2x_0} av_0)^2-1}}+\ord{\sfrac{\eps^3}{x_0^3}}\;,
\end{align}
consistent with the assumption that $T=\ord{\eps/x_0}$.
Taking the second row of \cref{sloww} as
\begin{align}\label{sloweta}
\dot \eta&=\sfrac{\eps}{x_0} v_0+\ord{\sfrac{\eps^2}{x_0^2}}\;,
\end{align}
we can now integrate to find the change in $y$ (equivalently $\eta$) over the time $T$ to be
\begin{align}
y_k&=y_{k-1}+\sfrac{\eps}{x_0} v_{k-1}T+\ord{\sfrac{\eps^3}{x_0^3}}\nonumber\\
&=y_{k-1}+\frac{4\eps^2 v_{k-1}}{x_0^2\sqrt{(y_{k-1}+\sfrac{\eps}{2x_0} av_0)^2-1}}+\ord{\sfrac{\eps^3}{x_0^3}}\;.
\end{align}
Lastly $\omega_k=\omega_{k-1}-2\pi Y_\pm$ translates into
\begin{align}
v_k=v_{k-1}-4Y_\pm\;,
\end{align}
completing the map in \cref{stairmap}-\cref{Tmap}.
\end{proof}

\bigskip
\section{Types of cycle in the nonlinear system}\label{sec:bigsmall}

Let us now piece together the three main features of the dynamics found in \cref{sec:slowfast}, namely the slow small arcs from \cref{thm:smallarc}, the slow large arcs from \cref{thm:bigarc}, and the fast staircases from \cref{thm:stairs}.

We will assume that the slow dynamics on $\op M_\eps$, and the fast dynamics outside $\op M_\eps$, can be concatenated simply at the turning points $\op C_m^\pm$. 
More precisely:

 \begin{assumption}[Matching of slow and fast flows]\label{ass1}
The matching between orbits of the slow flow given by \cref{thm:Me_arcs}, and the fast staircases given by \cref{thm:stairs}, in the system \cref{slow}, is equivalent to concatenating the two flows simply at the turning points $\op C_m^\pm$, up to higher order perturbations.
 \end{assumption}


To prove this requires a lengthy but straightforward application of well established methods. It consists of matching between the slow and fast flows, of which we have derived only the leading order asymptotics in \cref{sec:slowfast}, 
using standard geometric singular perturbation theory \cite{f79,j95}, in particular that related to folds in slow manifolds and local equivalence to a Riccati equation \cite{rosov,krupa05}.

Since we have taken the simple (piecewise linear) ramp function \cref{sign} for $\lambda$, the hyperbolicity of the slow manifolds $\op M_\eps$ holds all the way up to the boundaries of the layer $z=\pm\eps$ (though not in $|z|>\eps$). If we had taken a smooth function for $\lambda$ in terms of $z$, as is common in studies of the regularization of nonsmooth systems, e.g. \cite{j20malt,bonet2016,hk15}, then matching would also be necessary to determine how slow or fast solutions behave as the approach the boundary $|z|=\eps$, because normal hyperbolicity of $\op M_0$ would be lost in a neighbourhood of $|z|=\eps$. Similar to proving \cref{ass1}, this would be a lengthy but straightforward application of well established methods, one that we avoid here.

The precise results of such matching will not affect the qualitative features of the flow that we identify in this section, so we will formulate our main results that follow only qualitatively, leaving more rigorous statements to perhaps be formulated in future work.

For this section we again return to the coordinate system $(x,y,u)$ rather than using $v=xu$.

\subsection{Large cycles}\label{sec:bigcycles}

\begin{theorem}[Large cycles]\label{thm:bigcycles}
For large enough $x$ and under \cref{ass1}, there exist large cycles that consist of slow arcs given by \cref{thm:bigarc}, which travel between $y=-1$ and $y=+1$ in either direction, connected by fast staircases given by \cref{thm:stairs}. The $u$-amplitude of these cycles shrinks with each new cycle.
\end{theorem}

These large cycles are simply formed by concatenating slow large arcs given by \cref{thm:bigarc}, with the fast staircases given by \cref{thm:stairs}, as illustrated in \cref{fig:bigcycscales}. Under \cref{ass1} the cycles so obtained correspond to leading order to solutions of the full system \cref{slow}.
\begin{figure}[h!]\centering
\includegraphics[width=0.98\textwidth]{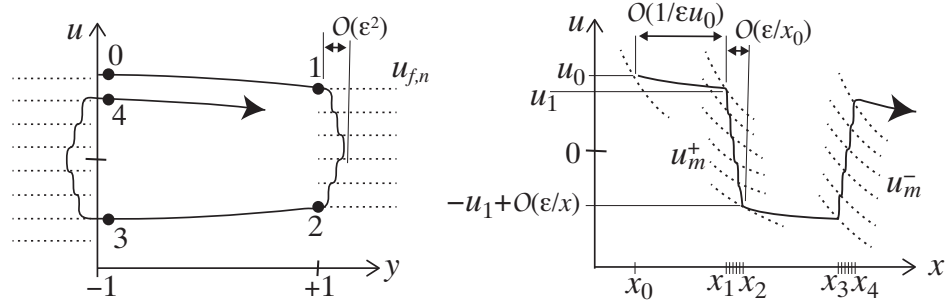}
\vspace{-0.3cm}\caption{\small\sf Large cycles formed by concatenating the large arcs of \cref{thm:bigarc} with the fast staircases of \cref{thm:stairs}, shown in $(y,u)$ and $(x,u)$.}\label{fig:bigcycscales}\end{figure}

\begin{proof}[Proof of \cref{thm:bigcycles}]

Let us consider an arc that starts from a point $(x_0,y_0,u_0)$$\in\op C_m^-$, evolves via a large arc to $(x_1,y_1,u_1)\in\op C_m^+$, then via a staircase to $(x_2,y_2,u_2)\in\op C_m^+$, then via a second large arc to $(x_3,y_3,u_3)\in\op C_m^-$, and completes the cycle with a staircase to $(x_4,y_4,u_4)\in\op C_m^-$. Such a cycle is illustrated in \cref{fig:bigcycscales}.
Then by definition the $y$-coordinates are $y_0=y_4=y_3=-1$ and $y_1=y_2=+1$, and the other coordinates are as follows.

Starting from $(x_0,-1,u_0)$ we apply \cref{thm:bigarc} to reach $(x_1,y_1,u_1)$, then \cref{thm:stairs} to reach $(x_2,y_2,u_2)$, then \cref{thm:bigarc} again to reach $(x_3,y_3,u_3)$, and finally  \cref{thm:stairs} once more to reach $(x_4,y_4,u_4)=(x_3,-1,-v_3)$. This gives, to leading order in $\eps/x_0$ and $u_0$,
\begin{align*}
(x_1,y_1,u_1)&\approx(x_0+\sfrac{2}{|u_0|\eps},+1,u_0-\sfrac{2(2+u_0)}{|u_0|\eps x_0})\;,\nonumber\\
(x_2,y_2,u_2)&=(x_1,+1,-u_1)\;,\nonumber\\
(x_3,y_3,u_3)&\approx(x_2+\sfrac{2}{|u_2|\eps},-1,u_2-\sfrac{2(2+u_2)}{|u_2|\eps x_2})\;,\nonumber\\
(x_4,y_4,u_4)&=(x_3,-1,-u_3)\;,
\end{align*}
and putting these together we have the return map to the section $$\cc{y=-1,\;u>0}$$ given by
\begin{align}\label{bigmap}
(x_4,y_4,u_4)
	&=\bb{x_0+\sfrac{4}{\eps u_0}+\ord{\sfrac1{x_0}}
	, -1,
	u_0-\sfrac{4}{\eps x_0}+\ord{\sfrac1{x_0^2}}	}\;.
\end{align}
Thus these {large cycles} are dissipative, returning to $y=-1$ with a $u$-coordinate that shrinks by an amount $\sfrac4{\eps x_0}$ with each subsequent cycle.
For this to happen it is important to notice that indeed $u$ decreases uniformly from positive to negative value in the map \cref{bigmap}, which need not be the case if, for example, \cref{bigmap} gave $u_4=u_0-\sfrac{4}{\eps x_0^r}$ for $r>1$, in which case $u$ could tend towards a positive limit. To show that indeed $u$ decreases without bound, introduce a new variable $v=1/xu$, for which \cref{bigmap} gives $v_4=v_0/(1-v_0^2)$ to leading order. From small positive values this grows until it passes $v=1$, then becomes negative which, since $x$ is strictly positive, corresponds to $u$ passing through zero. 
\end{proof}

As $u$ shrinks, eventually it will satisfy \cref{arclim}, that is $u_0^2< u_{0,c}^2\approx\frac{8}{\eps x_0}$, and the orbit must then transition to a small cycle.

\subsection{Small cycles}\label{sec:smallcycles}

\begin{theorem}[Small cycles]\label{thm:smallcycles}
For large enough $x$ and under \cref{ass1}, there exist small cycles consisting of slow arcs given by \cref{thm:smallarc}, starting and ending at $y=-1$, connected by fast staircases given by \cref{thm:stairs}. The $u$-amplitude and $y$-amplitude shrinks of these shrinks with each new cycle.
\end{theorem}

Typically a small arc is followed by a staircase, followed by another small arc, then another staircase, and so on, forming small cycles in this way, as illustrated in \cref{fig:smallcycscales}. As in \cref{thm:bigcycles}, under \cref{ass1} these correspond to leading order to solutions of the full system \cref{slow}. Moreover we can show that these are dissipative, similarly to the large cycles, by iterating the maps from \cref{thm:smallarc} and \cref{thm:stairs}.
%
\begin{figure}[h!]\centering
\includegraphics[width=0.98\textwidth]{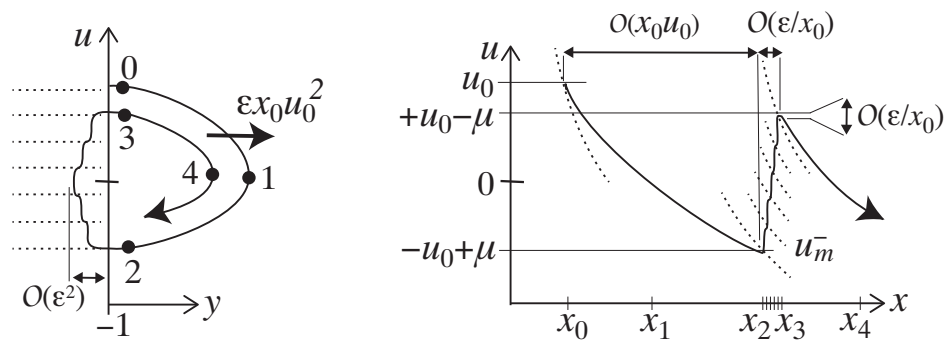}
\vspace{-0.3cm}\caption{\small\sf Small cycles formed by concatenating the large arcs of \cref{thm:smallarc} with the fast staircases of \cref{thm:stairs}, shown in $(u,y)$ and $(u,x)$, with the cycle shrinking by an amount $\mu=\sfrac23u_0^2$ each time. }\label{fig:smallcycscales}\end{figure}

\begin{proof}[Proof of \cref{thm:smallcycles}]
Take a small arc with initial condition $(x_0,-1,u_0)$, that passes through $u=0$ at $(x_1,y_1,0)$, returns to $y=-1$ at $(x_2,-1,u_2)$, then evolves through a staircase to $(x_3,-1,u_3)$, and finally passes through $u=0$ again at $(x_4,y_4,0)$. Such a cycle is illustrated in \cref{fig:smallcycscales}.

By \cref{harcu} the arc starting from $(x_0,-1,u_0)$ has a $y$-amplitude $\Delta y=1+y_1\sim\sfrac14\eps x_0u_0^2$, and by \cref{smallarcs} returns to coordinate $$\bb{x_2,-1,u_2}=\bb{x_0(1+u_0),-1,-u_0+\sfrac23u_0^2}\;.$$ Now this evolves through a staircase according to \cref{sec:stairs}, to
\begin{align}
\bb{x_3,-1,u_3}=\bb{x_0(1+u_0),-1,u_0-\sfrac23u_0^2}\;,
\end{align}
from which it is clear that $u$ decreases with each successive cycle.
This then connects to another slow arc, for which we again apply \cref{harcu} to show that this has a $y$-amplitude
$\Delta y=1+y_4\sim\sfrac14\eps x_3u_3^2
=\sfrac14\eps x_0u_0^2(1-\sfrac13u_0+\ord{u_0^2})$.

Therefore, to leading order the $y$-amplitude $\Delta y$ shrinks as
\begin{align}\label{peaksh}
\Delta y\mapsto(1-\sfrac13u_0)\Delta y+\ord{u_0^2}\;.
\end{align}

\end{proof}

Note that the time between adjacent peaks is
$x_4-x_1=x_3(1+\hf u_3)-x_1=x_2(1-\hf u_2)-x_1
=x_0u_0(1+\sfrac16u_0+\ord{u_0^2})$.
\Cref{fig:shrinkers} verifies numerically that this result holds for sufficiently large $x$.
\begin{figure}[h!]\centering
\includegraphics[width=0.85\textwidth]{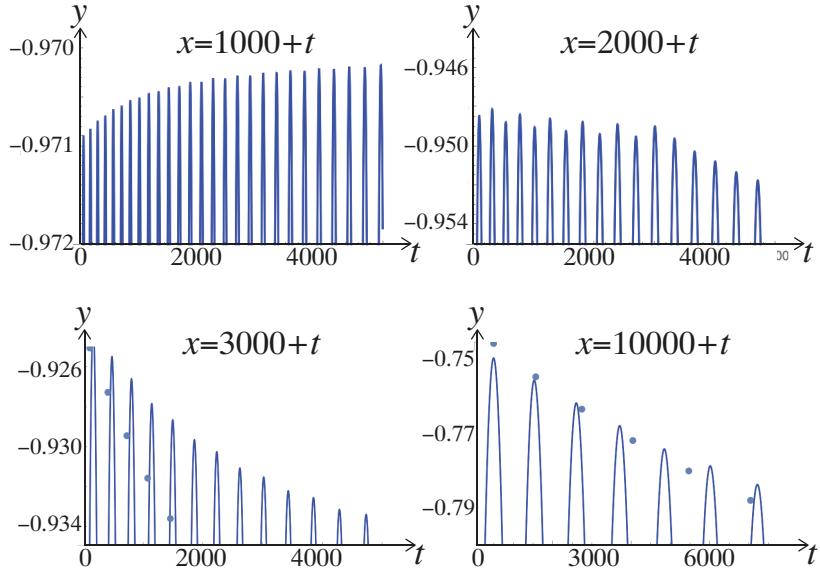}
\vspace{-0.3cm}\caption{\small\sf Plots showing the result that small arcs begin to shrink at a rate $\sfrac13$ but {\it only for large enough $x$}, shown here from $x\sim10000$. The plots show $y$ against $t$ at time $x_0+t$, from $y_0=-1$, $u_0=0.1$, with $x_0=1000,2000,3000,10000$. We only show the peaks of the arcs in $y\in[-1,+1]$, but each trajectory returns to $y=-1$ between peaks. In the last two pictures only, the dots show the predicted peaks from \cref{peaksh}, which the last picture shows forms a good approximation for large $x$. }\label{fig:shrinkers}\end{figure}

It is sensible to ask, since these small cycles shrink towards $(y,u)=(-1,0)$ with each iteration, whether they decrease to zero. However, it is clear from the system \cref{slow} that $(-1,0)$ is not an invariant, indeed since $\dot u$ keeps oscillating at $u=0$ between $0<\dot u<2$ as $x$ increases, it is surprising that solutions can shrink to $(y,u)=(-1,0)$ at all. Yet indeed by the argument above, the small cycles can only continue to shrink, and they do so by means of jumping between the slow arcs and fast staircases.

So let us estimate their size. Near $u=0$ the fast staircase must occur between folds of the slow manifold $\op M_\eps$ that are a distance of order $1/x$ apart, so the $u$-width of the cycles shrinks as $\Delta u\sim1/x$. The $y$-amplitude therefore shrinks as $\Delta y\sim\sfrac14\eps xu^2\sim\sfrac14\eps/x$. Hence the cycles occupy a region $(y,u)=(-1,0)+\sfrac1x(\eps,1)$ as $x\rightarrow\infty$, despite there being no invariant of the system at $(y,u)=(-1,0)$ to tend towards!

\subsection{The sliding-escaping boundary}\label{sec:escape}


Orbits that enter the switching layer $u\in[-1,+1]$ may either pass directly through the layer and exit from the opposite side (thus crossing from $z>+\eps$ to $z<-\eps$ or vice versa), or they may evolve onto a slow manifold $\op M_\eps$ in the region $|y|\le1$. In the latter case those orbits may remain in the slow-fast dynamics of the layer $u\in[-1,+1]$ (corresponding to sliding in the discontinuous system) for some finite time, or they may become trapped inside it for all future time.
\begin{theorem}[Trapping region]\label{thm:trap}
For large enough $x$ and under \cref{ass1}, orbits that enter the switching layer $|z|\le\eps$ from $|z|>\eps$ with $y$ values in the interval $[\sfrac13+\ord{\sfrac\eps x},1]$ can escape from the layer back into motion in $|z|>\eps$, while those entering sliding with $y\in[-1,\sfrac13+\ord{\sfrac\eps x}]$ become trapped in sliding for all later times.
\end{theorem}
Note from \cref{sec:bigcycles} and \cref{sec:smallcycles} that the trapped orbits will first form large cycles of shrinking amplitude, then eventually small arcs of shrinking amplitude, contracting gradually towards smaller $u$.


\begin{proof}[Proof of \cref{thm:trap}]
Consider an orbit that enters the switching layer at $(x_0,y_0,u_0)$ where $u_0=+1$, and with $|y_0|<1$ so the orbit enters the sliding region (the region of $y$ values for which the critical manifold $\op M_0$ exists). We have to show that if $y_0\lesssim1/3$ this orbits becomes trapped in the layer $|u|<1$, and otherwise it escapes the layer via $u=-1$.

For large $x$, this orbit encounters a branch of the slow manifold $\op M_\eps$ within a distance $\ord{\eps/x_0}$ of $u=+1$, so let us say that $(x_0,y_0,u_0)\in\op M_\eps$ within $\ord{\eps/x_0}$. The orbit then evolves via the slow flow \cref{arcapprox}, and let us assume it reaches a coordinate $(x_1,y_1,u_1)$ with $y_1=+1$. It then enters a staircase and by \cref{wholestairs} evolves to a coordinate $(x_2,y_2,u_2)$ where $y_2=+1$, and is attracted onto another branch of $\op M_\eps$ at $u=u_2+\ord{\eps/x_0}$. Finally this evolves according to the slow flow \cref{arcapprox} again to some coordinate $(x_3,y_3,u_3)$. If $u_3=-1$ and $y_3>-1$ then the orbit escapes from the layer. If $y_3=-1$ and $u_3>-1$ then the orbit enters another staircase, and by the dissipation shown in \cref{sec:bigcycles} the orbit is trapped in $|u|<1$.

The delineating case therefore is where $u_3=y_3=-1$. So let us consider this delineating orbit, which travels from $(x_0,y_0,+1)$, via the slow flow \cref{arcapprox} to $(x_1,+1,u_1)$, via a staircase to $(x_1,+1,-u_1)+\ord{\eps/x_0}$, and again via the slow flow \cref{arcapprox} to $(x_3,-1,-1)$. Applying \cref{arcapprox} to the first and last steps we have
\begin{subequations}\label{esc}
\begin{align}
 y_1&=y_0+2\eps(x_0-x_1)+\eps x_0(u_0+2)\log\frac{x_1}{x_0}+\ord{\eps/x_0}\;,\label{esc1}\\
 u_1x_1&=u_0x_0+2(x_0-x_1)+\ord{\eps/x_0}\;,\label{esc2}\\
 y_3&=y_2+2\eps(x_2-x_3)+\eps x_2(u_2+2)\log\frac{x_3}{x_2}+\ord{\eps/x_0}\;,\label{esc3}\\
 u_3x_3&=u_2x_2+2(x_2-x_3)+\ord{\eps/x_0}\;.\label{esc4}
 \end{align}
 \end{subequations}
 Substituting $y_1=y_2=-y_3=u_0=-u_3=+1$ by our boundary conditions (i.e. the locus of the trajectory described above), with $x_2=x_1$ and $u_2=-u_1$ for the staircase by \cref{wholestairs}, we are then left with four equations for $y_0$ in terms of $x_0,x_1,x_3,u_1$. We can solve \cref{esc2} and \cref{esc4} for $u_1$ and $x_3$ in terms of $x_0,x_1$, as
\begin{align}
u_1=1+\sfrac{3x_0-3x_1}{x_1}\;,\qquad x_3=4x_1-3x_0\;.
 \end{align}
 If we let $\tau=(x_1-x_0)/x_0$, then \cref{esc1} and \cref{esc3} become
\begin{subequations}
\begin{align}
 y_0&=1+2\eps x_0\tau-3\eps x_0\log(1+\tau)+\ord{\eps/x_0}\;,\\
 2&=6\tau\eps x_0-(1+4\tau)\eps x_0\log\sfrac{1+4\tau}{1+\tau}+\ord{\eps/x_0}\;,
 \end{align}
 \end{subequations}
 and approximating for small $\tau$ we can solve these to find
\begin{subequations}
\begin{align}
\tau&=\sfrac2{3\eps x_0}+\ord{1/\eps^2x_0^2}\;,\\
y_0&=\sfrac13+\ord{1/\eps^2x_0^2}\;,
 \end{align}
 \end{subequations}
proving the result.
\end{proof}

The implication is that, for large $x$, if an orbit enters sliding with $y_0\in\bb{\sfrac13+\ord{\sfrac1{\eps x_0}},1}$ then it may escape back into crossing. \Cref{fig:boundary} verifies the result with $x_0=1000$.

\begin{figure}[h!]\centering
\includegraphics[width=0.65\textwidth]{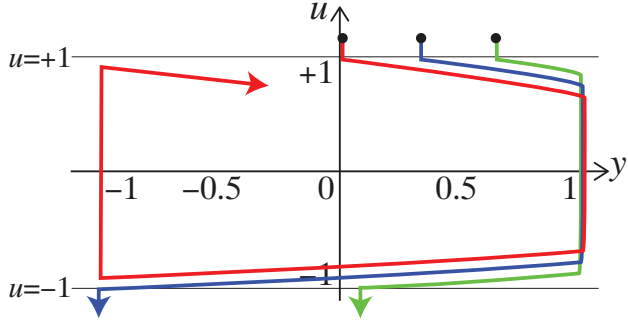}
\vspace{-0.3cm}\caption{\small\sf Simulations in the $(y,u)$ plane illustrating \cref{thm:trap}. From an initial condition $(1000,y_0,+1)$, an orbit it simulated with: $y_0=0$ (red), $y_0=1/3$ (blue), $y_0=2/3$ (green). }\label{fig:boundary}\end{figure}

Since orbits can only enter the layer at $u=+1$ where $\dot u<0$, and $\eps\dot u=-\eps a-y-\sin\sq{\sfrac32\pi x}$, we can say more completely that an orbit starting at $(x_0,y_0,+1)$ will become trapped if (to leading order)
\begin{align}
1/3>y_0>&-\eps a-\sin\sq{\sfrac32\pi x}\;,
\end{align}
and otherwise will escape the layer after only briefly sliding (with just one fast staircase) if
\begin{align}
y_0>&\max\bb{1/3,\;-\eps a-\sin\sq{\sfrac32\pi x}}\;,
\end{align}
to order $\eps/x$.

\section{Remarks on the discontinuous system, $\eps=0$}\label{sec:periodics}


Letting $\eps\rightarrow0$, the system \cref{sys} with either of the switching rules in \cref{nonlin} has two regions of smooth dynamics,
\begin{subequations}
\begin{align}
R_+&=\cc{(x,y,z)\in\mathbb R^3\;:\;z>0}\quad{\rm where}\quad \omega=\omega_+\;,\\
R_-&=\cc{(x,y,z)\in\mathbb R^3\;:\;z<0}\quad{\rm where}\quad \omega=\omega_-\;,
\end{align}
\end{subequations}
separated by a switching threshold $\surf$,
\begin{align}\qquad\;\;\;
\surf&=\cc{(x,y,z)\in\mathbb R^3\;:\;z=0}\quad{\rm where}\quad \omega\in[\omega_-,\omega_+]\;.
\end{align}

Solutions to \cref{sys0} evolve through $R_+$ and $R_-$, and can either {\it cross} transversally through $\surf$ between $R_+$ and $R_-$, or else can {\it slide} along $\surf$ for certain intervals of time. In \cref{sec:sliding} we describe the regions on $\surf$ where sliding can occur, in \cref{sec:crossing} we show that there exist periodic orbits that cross $\surf$. This section is somewhat cursory, because as we will see, standard concepts for $\eps=0$ cannot capture the rich behaviour we have derived for $\eps\ge0$ in \cref{sec:slowfast} to \cref{sec:bigsmall}.

\subsection{Sliding dynamics}\label{sec:sliding}

Taking the behaviour found in the switching layer $|z|\le\eps$ in \cref{sec:linear} to \cref{sec:bigsmall}, by letting $\eps\rightarrow0$ one observes what happens in the limit of the discontinuous system at $\eps=0$. Motion along invariant manifolds inside the layer $|z|\le\eps$ clearly becomes motion along the threshold $z=0$. This is known as {\it sliding} along $\surf$. Let us therefore compare our previous observations to what can be inferred by looking directly at the system with $\eps=0$ using the standard concept of sliding in discontinuous systems, which derive from Filippov \cite{f88}.


In essence, setting $\eps=0$ in \cref{sign} implies at $z=0$ only that the switching multiplier $\lambda$ takes a value in the interval $[-1,+1]$. To derive sliding motion, we seek a value of $\lambda$ in this interval that satisfies the equations $z=\dot z=0$ to give flow along $z=0$. These conditions can be used to derive the existence, stability, and the dynamics of sliding motion.

From the third row of \cref{sys} we first see that sliding can only occur for $|y|\le1$ (more precisely the problem $z=\dot z=0$ cannot be solved for $|y|>1$). Setting $z=\dot z=0$ we find that \cref{sys} reduces (for either function \cref{nonlin}) simply to sliding dynamics on $z=0$ given by
\begin{align}\label{slidesys}
\dot x=1\;,\quad
\dot y=0\;,\quad
\dot z=0\;.
\end{align}
So if an orbit hits the switching threshold at $(x,y,z)=(x_0,y_0,0)$, with $|y_0|\le1$, it will slide simply with a trajectory $x(t)=t+x_0$, $y(t)=y_0$, $z(t)=0$.

That trajectory may terminate if it reaches the boundary of the sliding region, where the vector field \cref{sys} is tangent to $z=0$ from either $z>0$ or $z<0$. That happens at
\begin{align}\label{Tpm}
\op T_\pm=\cc{(x,y,z)\in\surf\;:\; y=-\sin[\pi x(1\pm\hf)]=0}\;.
\end{align}
The sets $\op T_\pm$ are illustrated in \cref{fig:linslide}, and are identical to those in \cref{fig:typesl} and \cref{fig:types}.
\begin{figure}[h!]\centering
\includegraphics[width=0.9\textwidth]{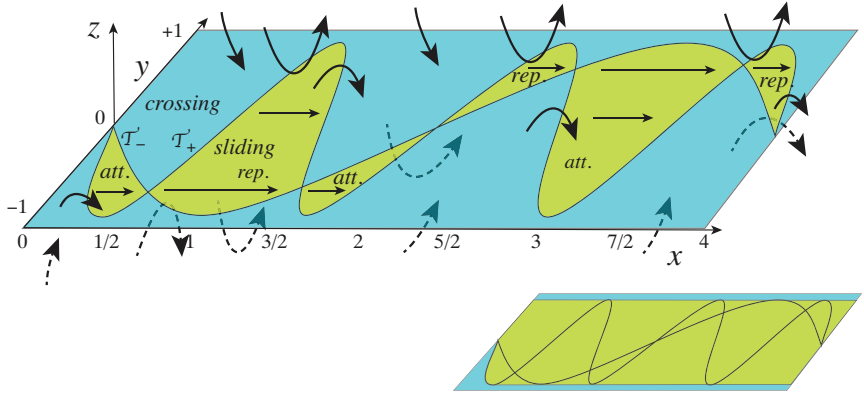}
\vspace{-0.3cm}\caption{\small\sf Sliding regions (green) and crossing regions (blue) of the linear system, bounded by the sets $\op T_+$ and $\op T_-$ at which the upper and lower vector fields are tangent to the switching threshold, respectively. The picture is the same for the nonlinear system, except that sliding occurs everywhere to the right of the curves $\op T_\pm$ as shown in the inset (bottom-right), as described below.}\label{fig:linslide}\end{figure}

This gives the impression that the sliding dynamics is simple, but this is true neither for the linear nor for the nonlinear system from \cref{nonlin}.
Regions of sliding can be attracting or repelling with respect to the dynamics outside $\surf$. This is easy to determine from the sign of $\sfrac{\partial\dot z}{\partial\lambda}$, namely if $\sfrac{\partial\dot z}{\partial\lambda}\!<\!0\;[>\!0]$ then the sliding region is attractive [repulsive]. 
The regions of attracting or repelling sliding differ between the nonlinear and linear systems \cref{non} and \cref{lin}.

For the linear system, taking \cref{lin} and solving $z=\dot z=0$ for the switching multiplier $\lambda\in[-1,+1]$, we find it takes a value of $\lambda=\sfrac{\sin(\pi\omega_+x)+\sin(\pi\omega_-x)}{\sin(\pi\omega_+x)-\sin(\pi\omega_-x)}$.
By looking at the sign of
$$\sfrac{\partial\dot z}{\partial\lambda}=\hf\cc{\sin(\hf\pi x)-\sin(\sfrac32\pi x)}\;,$$ we see that the sliding regions are
\begin{subequations}\label{linstab}
\begin{align}
{\rm attracting\;for}\quad&x\in\bb{0,\hf}\cup\bb{\sfrac32,2}\cup\bb{\sfrac52,\sfrac72}+4m\;,\quad m\in\mathbb Z\;,\\
{\rm repelling\;for}\quad&x\in\bb{\hf,\sfrac32}\cup\bb{2,\sfrac52}\cup\bb{\sfrac72,4}+4m\;,\quad m\in\mathbb Z\;,
\end{align}
\end{subequations}
on $y\in\bb{-\sin[\pi x/2],-\sin[3\pi x/2]}$, shown in \cref{fig:linslide}. These concepts are all an entirely standard application of Filippov's methods, see e.g. \cite{f88}, so we will not discuss them in greater detail. It is quite straightforward to see that these regions correspond to the attracting and repelling branches of the invariant manifold $\op M_0^{\rm lin}$ in \cref{sec:linear}. Of particular note are the so-called {\it two-fold} singularities that occur where $\op T_\pm$ intersect, and these correspond to the non-hyperbolic sets $\op C_{mn}$ in \cref{sec:linear}; for general theory of two-folds see for example chapter 13 of \cite{j18book}, also \cite{t18in3d,t07}, and references therein.

For the nonlinear system, taking \cref{non} and solving $z=\dot z=0$ for the switching multiplier $\lambda$, we now find that there are multiple solutions for $\lambda$, and that unlike in the linear system where the sliding regions are confined between $\op T_\pm$, here we can find sliding solutions for any $|y|\le1$ after some time $x$, as shown in the inset of \cref{fig:linslide}. The sign of $$\sfrac{\partial\dot z}{\partial\lambda}=-\hf\pi x\cos[\pi x(1+\hf\lambda)]$$
indicates that different solutions for $\lambda$ give either attracting or repelling sliding solutions,
\begin{subequations}\label{nonstab}
\begin{align}
{\rm attracting\;for}\quad&\lambda\in\bb{\sfrac{3}x-2,\sfrac{5}x-2}+\sfrac{4m}x\;,\quad m\in\mathbb Z\;,\\
{\rm repelling\;for}\quad&\lambda\in\bb{\sfrac{1}x-2,\sfrac{3}x-2}+\sfrac{4m}x\;,\quad m\in\mathbb Z\;,
\end{align}
\end{subequations}
on $|y|<1$. It is not possible to make sense of these different solutions here in the limit $\eps=0$, instead we must consider $\eps\ge0$ as we did in \cref{sec:slowfast}, but we can see that these regions are at least consistent with the attracting and repelling branches of the invariant manifold $\op M_0$ in \cref{sec:slowfast}. The multiple overlapping branches of sliding are typical when there is nonlinear dependence on the discontinuous quantity $\lambda$ as in \cref{non}, and such analysis was set out in general in \cite{j18book}. 


Let us now turn to the dynamics that lies outside $|z|\ge\eps$ and crosses through the discontinuity transversally, which we have not touched upon in the previous sections at all, taking again the limiting case $\eps=0$.

\subsection{Crossing dynamics}\label{sec:crossing}

For $z\neq0$, the solution to \cref{sys} from an initial point $y(0)=y_0$, $z(0)=z_0$, is
\begin{align}\label{crossol}
y(t)&=\sfrac1{\beta\mu}\Big(e^{-\hf at}S(t)+\mu R(t)\Big)\;,
\end{align}
with $x(t)=x_0+t$ and $z(t)=y'(t)$, where
\begin{align*}
S(t)&=\mu Q\cos(\hf\mu t)+P\sin(\hf\mu t)\;,\nonumber\\
R(t)&=a\pi\omega_\pm\cos(\pi\omega_\pm t)+\gamma\sin(\pi\omega_\pm t)\;,\nonumber\\
P&=aQ+2(\beta z_0-\gamma\pi\omega_\pm)\;,\nonumber\\
Q&=\beta y_0-a\pi\omega_\pm\;,\qquad
\beta=\gamma^2+a^2\pi^2\omega_\pm^2\;,\nonumber\\
\gamma&=\pi^2\omega_\pm^2-1\;,\qquad \mu=\sqrt{4-a^2}\;,\nonumber
\end{align*}
and with the appropriate $\pm$ signs being taken for $\sign(z)=\pm1$. 
If it is possible to concatenate a solution arriving from one side of $z=0$ with one departing from the other, preserving the direction of time, we obtain a solution that {\it crosses} the switching threshold $\surf$. This is the usual understanding according to the theory of Filippov systems \cite{f88}.

Simulations reveal that the system has a number of invariant objects formed from such crossing dynamics. Some straightforward simulations reveal two key features in the crossing dynamics, namely periodic orbits, which for $a>0$ are attractors, as shown in \cref{fig:po}(i), and which for $a=0$ are surrounded by invariant tori, one example of which is shown in \cref{fig:po}(ii). The periodic orbits have period $\Delta t=8$, and for each orbits there is another, identical in $(y,z)$ but shifted by $\Delta t=4$, as a consequence of the fact that the system with $\omega_+=3/2$ and $\omega=1/2$ is 4 periodic.
\begin{figure}[h!]\centering
\includegraphics[width=0.98\textwidth]{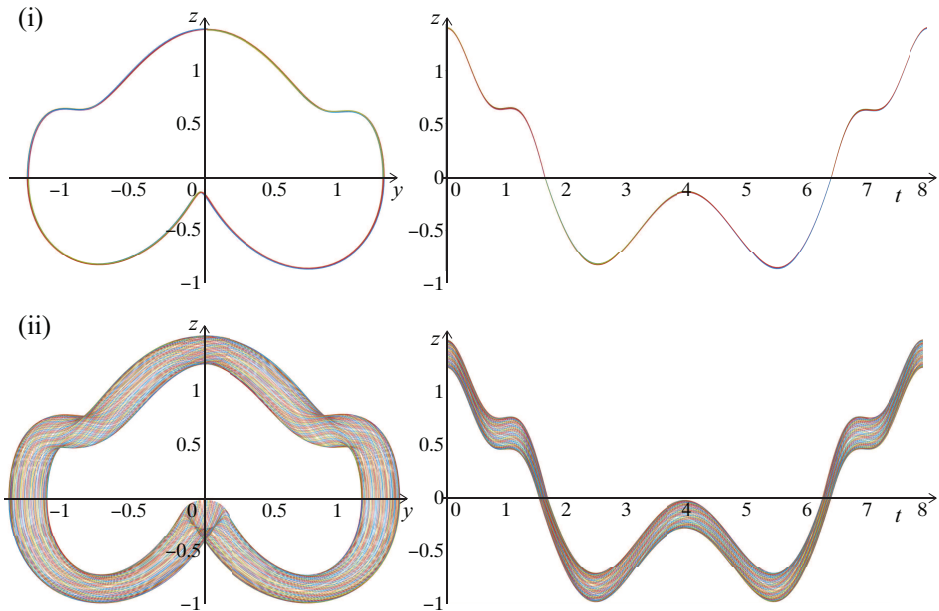}
\vspace{-0.3cm}\caption{\small\sf Examples of invariant objects in the crossing dynamics, simulated from an initial condition $(x_0,y_0,z_0)=(0,0.02,1.3)$, shown in $(y,z)$ and $(x,z)$ projections. (i) For $a=0.01$ we have an attracting periodic orbit. (ii) For $a=0$ we have an orbit on an invariant torus, a continuum of which surround a periodic orbit similar to that in (i). For every periodic orbit there exists a second orbit, identical in $(y,z)$ but shifted in $x$ by $\Delta x=4$. }\label{fig:po}\end{figure}

One can extend this analysis to the singular perturbation problem for small $\eps>0$, and will find a small perturbation of the periodic orbits described above. The proof of this is a direct extension of the results in \cite{j20malt}, so we will not repeat them here. 
In short, when the linear system given by \cref{sys} with \cref{lin} is regularized, the crossing periodic orbits persist. In contrast, when $\eps>0$ in the nonlinear system given by \cref{sys} with \cref{non}, no exactly periodic orbits can persist, because (as we will show below) the system in $|z|<\eps$ changes with $x$ in a non-periodic fashion. Instead, extending arguments from \cite{j20malt}, it is clear that
an orbit exists that, at least for a time of order $1/\eps$, lies within an $\eps$-neighbourhood of a curve that corresponds to each of the periodic orbits found above, but is not an exact solution of the $\eps>0$ system.
Moreover, for $0\ll\eps\ll a$ there will exist an attractor within an $\eps$-neighbourhood of the crossing periodic orbits above. So we have the novel situation that solutions of the nonlinear system tend to the periodic orbits described above as time increases, but those asymptotic orbits are not actual solutions of the system; see \cite{j20malt} for a more complete description and proof.

Thus the linear and nonlinear systems differ in their crossing dynamics by only a small perturbation, in contrast to the sliding dynamics. 
There remains much to be studied here, but the precise nature of the crossing periodic orbits like those in \cref{fig:po} is beyond our present interest, firstly because they can be studied using standard theory, and secondly because they depend on the exact values of $\omega_\pm$ that we set in \cref{omega}. In particular there may exist periodic orbits that have both segments of sliding and crossing, which we have not considered here at all.

The brief description above at least shows that the system exhibits multi-stability, and as such, any differences in behaviour between the linear and nonlinear (or other) formulations, as indeed occur in the sliding dynamics, could have substantial consequences for the global behaviour of the system.

As one further remark for future interest, for $a=0$ the system is piecewise-Hamiltonian, and it is possible to find quantities that are conserved in either of the subsystems in $z>0$ or $z<0$, but which fail to be conserved when they intersect $z=0$. Nevertheless the existence of a continuum of invariant tori, such as that in \cref{fig:po}(ii), suggests that, local to the periodic orbits, an integrable Poincare map can be derived from \cref{crossol}. It is beyond our interest here to explore this further, certainly due to the presence of sliding in the region $\cc{z=0,\;|y|\le1}\subset\surf$, the full system is not Hamiltonian.

\bigskip
\section{Closing Remarks}\label{sec:conc}

The nonlinear or linear expressions in \cref{nonlin} are motivated by seeking a closed functional expression of the switch, rather than being derived from any physical laws, and so perhaps they do not accurately model any real switching mechanism, whether electronic, mechanical, or biological. Work to close this gap between such purely mathematical formulations and physical modeling is likely to continue, but as switching in physical and biological systems typically involves complex multi-scale and often uncertain processes, this problem is not trivial; indeed that is the reason why simplifications like \cref{nonlin} are widely used.

The point of the present work is merely to show how severe the effect of different formulations is, despite being consistent in the (of course singular) limit $\eps\rightarrow0$.
In the case of the oscillator \cref{sys0}, when the forcing is switched by abruptly ramping the frequency $\omega$ between $\omega_\pm$, the expression \cref{non} seems the more natural physical representation, but the model obtained is then not the more commonly adopted Filippov system, which instead takes the form of \cref{lin}. The crucial difference is in the dependence of the function $f$ on $\lambda$ in \cref{nonlin}, not in the shape of the switching function $\lambda$ in \cref{sign}. 

The values of $\omega_\pm$ used here are just for illustration, as are other details of the model. The importance of these results, and the preparatory analysis in \cite{j20malt}, lie not in the periodicity or precise forms of any particular attractors, or in characterizing any particular model, but rather in showing that the long term behaviour of a discontinuous system depends critically on the manner of its dependency on the discontinuity.



The nonlinear system (\cref{sys} with \cref{non}) and linear system (\cref{sys} with \cref{lin}) share certain properties, most notably the existence of crossing periodic orbits. They differ in the key fact that, while orbits in the linear system can only slide for short intervals of time (between the boundaries $\op T_\pm$ in \cref{fig:linslide}), in the nonlinear system we have seen from \cref{sec:bigsmall} that orbits can enter into sliding and become trapped there for all future time.

In their sliding behaviour, in fact, the two systems are vastly different. The linear system is $t=4$ periodic, and for $\eps=0$ becomes a Filippov system analyzable using standard methods, with its most novel feature being that it contains two-fold singularities at every $t=\hf+n$, $n\in\mathbb Z$, and so its perturbation to $\eps>0$ may also be interesting and non-trivial. The nonlinear system is not periodic in the switching layer, and as such, different forms of behaviour --- like the large and small cycles --- dominate at different times. These cycles are a new example of nonlinear sliding oscillations referred to as {\it jitter} in \cite{j18book,j18model}. There is clearly much more rich behaviour to be found in the nonlinear system also.

As an indication of the behaviour that might be found under closer inspection of the nonlinear system, if we fix $x=x_0$ and set $a=0$ in \cref{slow}, we obtain the undamped planar autonomous system
\begin{align}\label{auto}
\dot y&=\eps u\;,\nonumber\\
\eps\dot u&=-y-\sin\bb{\pi x_0(1+\hf u)}\;.
\end{align}
In fact it can be proven that, for order one times, the autonomous system \cref{auto} is a good approximation of the full system \cref{slow}, though the details are beyond our scope here. Let us just briefly describe the insight given by considering this autonomous approximation.

The system \cref{auto} has a unique equilibrium at $(y,u)=\bb{-\sin(\pi x_0),0}$, which is stable for $x_0\in(\sfrac32\pi,\sfrac52\pi)+2n$, and unstable for $x_0\in(\sfrac12\pi,\sfrac32\pi)+2n$, for any $n\in\mathbb N$.
So the change of stability occurs at 
\begin{align}\label{ct}
x_0=\hf+n\;,\qquad n\in\mathbb N\;,
\end{align}
at which the system becomes symmetric about $u=0$, since then $\eps\dot u=-y-\sin\sq{\pi x_0(1+\hf u)}=\cos\sq{\hf\pi x_0u}$, and as a result at times \cref{ct} the whole flow consists of a continuum of periodic orbits.
Passing through \cref{ct} the system undergoes an unusual form of synchronized canard explosion, an example of which is shown in \crefs{fig:canarde}{fig:canard}.
\begin{figure}[h!]\centering
\includegraphics[width=0.98\textwidth]{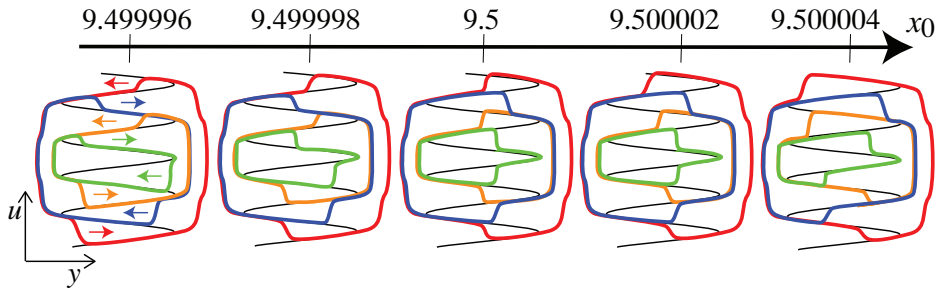}
\vspace{-0.3cm}\caption{\small\sf Galerie au canard: samples of the synchronized canard explosion through $x_0=9.5$, plotted with $\eps=1$, and $x_0$ varying between $9.5\pm0.000004$. The first and last images are also shown more clearly in \cref{fig:canard}. At $x_0=9.5$ itself (central image) 
and at any $x_0=\hf+n$ the entire flow consists of periodic orbits. The arrows in the leftmost image indicate the direction in which the branches are moving through the canard explosion as $x_0$ increases. }\label{fig:canarde}\end{figure}

Following the theme of the {\it chasse au canard} from the original work on canards \cite{b81}, an appropriate name for this phenomenon might be a {\it galerie au canard}. For any $x_0\neq\hf+n$ there exist a finite number, approximately $x_0/2$, of limit cycles formed by {\it large cycles} of the kind from \cref{sec:bigcycles}. Approximately half of these are [un]stable, and consist of two slow segments on [un]stable branches of $\op M_0$, connected by two fast staircases.
These limit cycles therefore form an (approximately) concentric family of relaxation oscillations. Just as the equilibrium changes stability at each $x_0=\hf+n$, each of these cycles also changes stability, and in a small neighbourhood of $x_0=\hf+n$ each cycle undergoes a canard explosion, containing segments that extend along both the stable and unstable branches of $\op M_0$. \Cref{fig:canard} shows an example as $x_0$ increases through $9.5$,
\begin{figure}[h!]\centering
\includegraphics[width=0.8\textwidth]{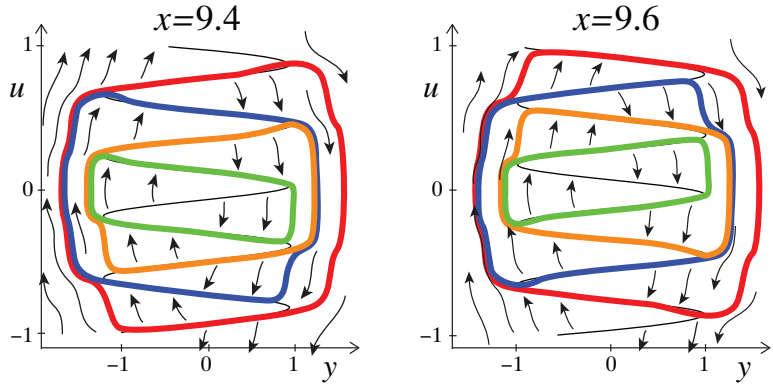}
\vspace{-0.3cm}\caption{\small\sf Images before and after the canard explosion at $x_0=9.5$, magnified to show the direction of the flow. Here we can see the stability and instability of adjacent periodic orbits, as how each orbit `walks' upward as $x_0$ increases by stepping either its left (red and orange unstable orbits here) or right (blue and green stable orbits here) side from one fold to the next.}\label{fig:canard}\end{figure}

We leave these and other rich behaviours to future interest.
The system \cref{sys} was proposed in \cite{j15hidden} to show how certain nonsmooth systems could be surprisingly challenging, analytically and numerically, to the extent that determining the `correct' way to express the system at the discontinuity became insurmountably difficult. In fact we see that simulating the system \cref{sys} numerically is not only difficult, for the nonlinear system it becomes literally impossible as $t$ increases, because timesteps would have to remain smaller than $\eps/t$ to capture the slow-fast character as the slow manifolds accumulate with density as $1/t$.

Because of this $1/t$ ageing, the main features we have uncovered here, as asymptotic results for $t\rightarrow\infty$, are beyond the scope of any numerical methods we are aware of. In fact, as basic numerical simulations confirm, the system becomes incredibly sensitive even for moderate values of $t$ (e.g. around $t\sim10$ for default numerical differential equation solvers in software such as Matlab or Mathematica). At all times, careful inspection is essential to verify that results are mathematically consistent with the inherent geometry, but at some time $t$ (not necessarily very large depending on numerical method, but always finite) numerical methods will fail to capture the solutions of \cref{sys} with the switching rule \cref{non} even approximately.

\bigskip
\noindent{\bf Acknowledgements}. The authors thank DD Novaes for helpful conversations regarding the discontinuous system.
PM was partially supported by the Spanish MINECO-FEDER Grant PID2021-123968NB-I00 and the Catalan grant 2017SGR1049.
This work is also supported by the Spanish State Research Agency, through the Severo Ochoa and Mar\'ia de Maeztu Program for Centers and Units of Excellence in R\&D (CEX2020-001084-M).
CB was partly supported by the Spanish MINECO-FEDER Grant PGC2018-098676-B-100 (AEI/FEDER/UE) and Catalan grant 2017SGR1049.
JMO was partially supported by the Spanish \emph{Agencia Estatal de Investigaci\'on} Project DPI2017-85404-P.

{\small
\bibliography{../../grazcat}

\begin{thebibliography}{10}

\bibitem{b81}
E.~Benoit, J.~L. Callot, F.~Diener, and M.~Diener.
\newblock Chasse au canard.
\newblock {\em Collect. Math.}, 31-32:37--119, 1981.

\bibitem{j20malt}
C.~Bonet-Rev\'es, M.~R. Jeffrey, P.~Martin, and J.~M. Olm.
\newblock Ageing of an oscillator due to frequency switching.
\newblock {\em CNSNS}, 102(105950):1--26, 2021.

\bibitem{bonet2016}
C.~Bonet-Rev\'es and T.~M. Seara.
\newblock Regularization of sliding global bifurcations derived from the local
  fold singularity of {F}ilippov systems.
\newblock {\em Discrete Contin. Dyn. Syst. Ser. A}, 36(7):3545--3601, 2016.

\bibitem{f79}
N.~Fenichel.
\newblock Geometric singular perturbation theory.
\newblock {\em J. Differ. Equ.}, 31:53--98, 1979.

\bibitem{f88}
A.~F. Filippov.
\newblock {\em Differential Equations with Discontinuous Righthand Sides}.
\newblock Kluwer Academic Publ. Dortrecht, 1988 (original in Russian 1985).

\bibitem{t18in3d}
Ot\'avio M.~L. Gomide and M.~A. Teixeira.
\newblock Generic singularities of 3{D} piecewise smooth dynamical systems.
\newblock {\em Advances in Mathematics and Applications}, pages 373--404, 2018.

\bibitem{j15hidden}
M.~R. Jeffrey.
\newblock The ghosts of departed quantities in switches and transitions.
\newblock {\em SIAM Review}, 60(1):116--36, 2017.

\bibitem{j18book}
M.~R. Jeffrey.
\newblock {\em Hidden Dynamics: The mathematics of switches, decisions, \&
  other discontinuous behaviour}.
\newblock Springer, 2019.

\bibitem{j18model}
M.~R. Jeffrey.
\newblock {\em Modeling with nonsmooth dynamics}.
\newblock Frontiers in Applied Dynamical Systems. Springer Nature Switzerland,
  2020.

\bibitem{j95}
C.~K. R.~T. Jones.
\newblock {\em Geometric singular perturbation theory}, volume 1609 of {\em
  Lecture Notes in Math. pp. 44-120}.
\newblock Springer-Verlag (New York), 1995.

\bibitem{hk15}
K.~U. Kristiansen and S.~J. Hogan.
\newblock On the use of blowup to study regularization of singularities of
  piecewise smooth dynamical systems in {R}3.
\newblock {\em SIADS}, 14(1):382--422, 2015.

\bibitem{rosov}
N.~Kh. Rosov and E.~F. Mischenko.
\newblock {\em Differential equations with small parameters and relaxation
  oscillations}.
\newblock Plenum, 2013.

\bibitem{st96}
J.~Sotomayor and M.~A. Teixeira.
\newblock Regularization of discontinuous vector fields.
\newblock {\em Proceedings of the International Conference on Differential
  Equations, Lisboa}, pages 207--223, 1996.

\bibitem{t07}
M.~A. Teixeira, J.~Llibre, and P.~R. da~Silva.
\newblock Regularization of discontinuous vector fields on {$R^3$} via singular
  perturbation.
\newblock {\em Journal of Dynamics and Differential Equations}, 19(2):309--331,
  2007.

\bibitem{krupa05}
S.~van Gils, M.~Krupa, and P.~Szmolyan.
\newblock Asymptotic expansions using blow-up.
\newblock {\em Zeitschrift \"ur angewandte Mathematik und Physik},
  56(3):369--97, 2005.

\end{thebibliography}
\bibliographystyle{plain}
}


\section*{Appendix A. Proofs of \cref{thm:Me} and \cref{thm:Me_arcs}}\label{sec:proofs}

Here we prove the two lemmas from \cref{sec:slowapprox}.

\begin{proof}[Proof of \cref{thm:Me}]
Local to the critical manifold $\op M_0$, on a region where it is normally hyperbolic, that is, away from the turning points \cref{fold},
there exist slow manifolds $\op M_\eps$ that are an $\sfrac\eps{x}$ perturbation of $\op M_0$.
Assume these can be expressed as \cref{Me} where
\begin{align}
y=Y(x,v;\eps):=Y_0(x,v)+\sfrac{\eps}{x}Y_1(x,v)+\sfrac{\eps^2}{x^2}Y_2(x,v)+...
\end{align}
As this is invariant it satisfies
\begin{align}
0&=\sfrac{d\;}{dt}\bb{Y(x,v;\eps)-y}=(\dot x,\dot y,\dot v)\cdot\nabla \bb{Y(x,v;\eps)-y}\nonumber\\
&=\bb{1,\;\sfrac{\eps}{x}v,\;(\sfrac1x-a)v-\sfrac{x}{\eps}y-\sfrac{x}{\eps}\sin\sq{\pi (x+\hf v)}}\cdot\nabla(Y(x,v;\eps)-y)\nonumber\\
%
&=-\sfrac{x}{\eps}(Y_0+\sin\sq{\pi (x+\hf v)})Y_{0,v}
\nonumber\\&\qquad
+Y_{0,x}+((\sfrac1x-a)v-Y_1)Y_{0,v}-(Y_0+\sin\sq{\pi (x+\hf v)})Y_{1,v}
\nonumber\\&\qquad
+\sfrac{\eps}{x}
\left\{Y_{1,x}-v  -\sfrac1xY_1+((\sfrac1x-a)v-Y_1)Y_{1,v}
	\right.\nonumber\\&\qquad\qquad\;\;\left.
	-(Y_0+\sin\sq{\pi (x+\hf v)})Y_{2,v}-Y_2Y_{0,v}\right\}
+...\nonumber\\
\Rightarrow\qquad
Y_0&=-\sin\sq{\pi (x+\hf v)}\;,\nonumber\\
Y_1&=\frac{Y_{0,x}}{Y_{0,v}}+(\sfrac1x-a)v=2+(\sfrac1x-a)v\;,\nonumber\\
Y_2&=\frac{Y_{1,x}-Y_1Y_{1,v}-\sfrac1xY_1+(\sfrac1x-a)v-v}{Y_{0,v}}\;,
\end{align}
thus the slow manifolds are approximated by $y=Y(x,v;\eps)$ as given by \cref{Ye}.
\end{proof}

%
\begin{proof}[Proof of \cref{thm:Me_arcs}]
Let
\begin{align}
h=Y(x,v;\eps)-y\;,
\end{align}
then the dynamics on the invariant manifold $\op M_\eps$ given by \cref{Me} satisfies
\begin{align}
0&=\sfrac{d\;}{dt}(Y(x,v;\eps)-y)\nonumber\\
&=-\pi(1+\hf\dot v)\cos\sq{\pi (x+\hf v)}\nonumber\\&\qquad\qquad
-\sfrac\eps{x}\cc{v+a\dot v+\sfrac2{x}+\sfrac1{x}(\sfrac2x-a)v}+\ord{\sfrac{\eps^2}{x^2}}\nonumber\\
\Rightarrow\qquad
\dot v&=-\sfrac{\pi\cos\sq{\pi (x+\hf v)}+\sfrac\eps{x}\cc{v(1+\sfrac2{x^2}-\sfrac1xa)+\sfrac2x}}
{\sfrac\pi2\cos\sq{\pi (x+\hf v)}+\sfrac\eps{x}a}+\ord{\sfrac{\eps^2}{x^2}}\nonumber\\
%
%
%
&=-2-2\sfrac{\eps}{\pi x}\frac{v(1+\sfrac2{x^2}-\sfrac1xa)+\sfrac2x+2a}{\cos\sq{\pi (x+\hf v)}}
+\ord{\sfrac{\eps^2}{x^2}}
\;.
\end{align}
Thus we have the slow dynamics \cref{M0proj}, and integrating with respect to time from an initial point $(x_0,y_0,v_0)$ gives \cref{arcapprox} directly.

\end{proof}

\end{document}